\documentclass[a4paper,11pt]{article}
\usepackage[utf8]{inputenc}

\usepackage{boxedminipage}
\usepackage{amsfonts}
\usepackage{amsmath} 
\usepackage{amssymb}
\usepackage{graphicx}
\usepackage{amsthm}
\usepackage{t1enc}
\usepackage{subfig}
\usepackage{cancel}
 \usepackage{textcmds}

\newtheorem{theorem}{Theorem}[section]
\newtheorem{lemma}[theorem]{Lemma}

\newtheorem{corollary}[theorem]{Corollary}
\newtheorem{definition}[theorem]{Definition}

\newtheorem{fact}[theorem]{Fact}
\newtheorem{question}{Question}
\newtheorem*{theorem*}{Theorem}
\newtheorem*{claim}{Claim}

\newcommand{\forceP}{\mathbb{P}}
\newcommand{\forceQ}{\mathbb{Q}}
\newcommand{\forceR}{\mathbb{R}}

\newcommand{\ZFC}{\mathsf{ZFC}}

\newcommand{\ZFP}{\mathsf{ZF}^-}

\newcommand{\CH}{\mathsf{CH}}

\newcommand{\PD}{\mathsf{PD}}

\makeatletter
\providecommand*{\capdot}{%
  \mathbin{%
    \mathpalette\@capdot{}%
  }%
}
\newcommand*{\@capdot}[2]{%
  \ooalign{%
    $\m@th#1\cap$\cr
    \hidewidth$\m@th#1\cdot$\hidewidth
  }%
}
\makeatother

\def\undertilde#1{\mathord{\vtop{\ialign{##\crcr
$\hfil\displaystyle{#1}\hfil$\crcr\noalign{\kern1.5pt\nointerlineskip}
$\hfil\tilde{}\hfil$\crcr\noalign{\kern1.5pt}}}}}

\title{Forcing the ${\Sigma^1_3}$-Separation Property}
\author{ Stefan Hoffelner\footnote{ WWU M\"unster. Research funded by thr Deutsche Forschungsgemsinschaft (DFG German Research Foundation) under Germanys Excellence Strategy EXC 2044 390685587, Mathematics M\"unster: Dynamics-Geometry-Structure. The author was additionally partially supported by FWF-GACR grant no. 17-33849L, Filters, ultrafilters and connections with forcing. He thanks S. D. Friedman for several discussions on a related topic. He thanks A. Lietz for discussions and several improvements, and F. Schlutzenberg, R. Schindler and L. Wu for more discussions. }}
\date{21.07.2022}

\begin{document}

\maketitle

\begin{abstract}
We generically construct a model in which the $\bf{\Sigma^1_3}$-separation property is true, i.e. every pair of disjoint $\bf{\Sigma^1_3}$-sets can be separated by a $\bf{\Delta^1_3}$-definable set. This answers an old question from the problem list $``$Surrealist landscape with figures$"$  by A. Mathias from 1968. We also construct a model in which the (lightface) $\Sigma^1_3$-separation property is true. 
\end{abstract}

\section{Introduction}

The separation property, together with the reduction property and the uniformization property, are three classical notions which were introduced and studied first by Polish and Russian descriptive set theorists in the 1920's and 1930's.

\begin{definition} Let $\Gamma$ be a (lightface or boldface) projective pointclass, and let $\check{\Gamma}=\{X \, : \, \omega^{\omega} \backslash X \in \Gamma \}$ denote the dual pointclass of $\Gamma$.
\begin{itemize}

\item  We say that  $\Gamma$ has the separation 
 property iff every pair $A_1$ and $A_2$ of disjoint elements of $\Gamma$ has a separating set $C  \in \Gamma \cap \check{\Gamma}$, where $C$ separates $A_1$ and $A_2$ if $A_1 \subset C$ and $A_2 \subset \omega^{\omega} \backslash C$.
\item  $\Gamma$ has the reduction property if for any pair $A_1$ and $A_2$ in $\Gamma$, there are disjoint sets $B_1 \subset A_1$ and $B_2 \subset A_2$ both in $\Gamma$ such that $A_1 \cup A_2= B_1 \cup B_2$.
\item $\Gamma$ has the uniformization property if  for every $A \subset \omega^{\omega} \times \omega^{\omega}$ there is a uniformizing function $f_A$ whose graph is in $\Gamma$, where we say that $f_A$ is a uniformizing function of $A$ if
$dom f_A = pr_1(A)=\{ x \in \omega^{\omega} \, : \, \exists y ((x,y) \in A) \}$ and $f_A \subset A$.
\end{itemize}

\end{definition}

It is rather straightforward to see that the uniformization property for $\Gamma$ implies the reduction property for $\Gamma$. A classical result due to Novikov shows that the reduction property can not hold simultaneously at both $\Gamma$ and $\check{\Gamma}$. Passing to complements immediately yields that the reduction property for $\Gamma$ implies that the dual $\check{\Gamma}$ has the separation property (see e.g. Y. Moschovakis book \cite{Moschovakis} for many more information on the orgin and history of these notions). Consequentially, $\bf{\Sigma^1_1}$ and $\bf{\Pi^1_2}$-sets have the separation property due to M. Kondo's theorem that $\bf{\Pi^1_1}$, hence also $\bf{\Sigma^1_2}$ has the uniformization property. The fact that the $\bf{\Sigma^1_1}$-separation property is true has been proved by N. Lusin already in 1927.  This is as much as $\ZFC$ can prove about the separation property. 

In G\"odel's constructible universe $L$ there is a good $\Sigma^1_2$-definable wellorder of the reals, hence the $\bf{\Sigma^1_n}$-uniformization property holds for $n\ge 3$, so $\bf{\Pi^1_n}$-separation must hold as well. On the other hand, by the celebrated results of Y. Moschovakis $\bf{\Delta}^1_{2n}$-determinacy implies the $\bf{\Pi}^1_{2n+1}$-uniformization property, so in particular under $\bf{\Delta}^1_2$-determinacy $\bf{\Sigma^1_3}$-separation holds.
Note here, that due to H. Woodin, ${\Delta^1_2}$-determinacy together with $\bf{\Pi}^1_1$-determinacy already implies that $M_1^{\#}$ exists and is $\omega_1$-iterable (see \cite{MSW}, Theorem 1.22) which in turn implies the existence of an inner model with a Woodin cardinal.

On the other hand, in the presence of $``$every reals has a sharp$"$, if the $\bf{\Sigma}^1_3$-separation property holds, then it does so because already $\bf{\Delta}^1_2$-determinacy holds. The above follows from Steel's and Woodin's solution to the fourth Delfino problem.  Steel, showed that in the presence of $``$every reals has a sharp$"$, the $\bf{\Sigma}^1_3$-separation property implies the existence of an inner model with a Woodin cardinal as well (see \cite{Steel}, Theorem 0.7), more precisely, under the stated assumptions, for any real $y$, there is a proper class model $M$ with $y \in M$, and an ordinal $\delta$ such that $V^M_{\delta+1}$ is countable and $\delta$ is a Woodin cardinal in $M$. Now results of Woodin which were later reproved by I. Neeman using different methods (see \cite{Neeman}, Corollary 2.3) the latter assertion implies that $\bf{\Delta}^1_2$-determinacy must hold in $V$.

It is natural to ask whether one can get a model of the $\bf{\Sigma^1_3}$-separation property from just assuming the consistency of $\ZFC$. Indeed, this question has been asked long before the connection of determinacy assumptions and large cardinals has been uncovered; it appears as Problem 3029 in A. Mathias's list of open problems compiled in 1968 (see \cite{Mathias},  or \cite{Kanovei1}, where the problem is stated again).
The problem itself seems to have a nontrivial history of attempted solutions (see \cite{Kanovei2} for an account).

Put in wider context, this paper can be seen as following a tradition of establishing consequences from (local forms of) projective determinacy using the methods of forcing. There is an extensive list of results which deal with forcing statements concerning the Lebesgue measurability and the Baire property of certain levels of the projective hierarchy. For the separation property, L. Harrington, in unpublished notes dating back to 1974, constructed a model in which the separation property fails for both $\bf{\Sigma}^1_3$ and $\bf{\Pi}^1_3$-sets. In the same set of handwritten notes, he outlines how his proof can be altered to work
for arbitrary $n \ge 3$. Very recently, using different methods, V. Kanovei and V. Lyubetsky devised a forcing which, given an arbitrary $n \ge 3$, produces a universe in which the $\bf{\Sigma}^1_n$- and  the $\bf{\Pi}^1_n$-separation property fails (see \cite{Kanovei1}).

Yet tools for producing models which deal with the separation property, the reduction property or the uniformization property in a positive way were non-existent. Goal of this paper is to show that the $\bf{\Sigma^1_3}$-separation property has no large cardinal strength, which answers Mathias question.

\begin{theorem*}
Starting with $L$ as the ground model, one can produce a set-generic extension $L[G]$ in which the $\bf{\Sigma^1_3}$-separation property holds.
\end{theorem*}
The proof method also allows to tackle the $\Sigma^1_3$-separation property:

\begin{theorem*}
Starting with $L$ as the ground model, one can produce a set-generic extension $L[G]$ in which the ${\Sigma^1_3}$-separation property holds.
\end{theorem*}

As always, the flexibility of the forcing method can be exploited to produce effects which can not be inferred from projective determinacy assumptions alone. An example would be that the above proofs lift without pain to statements about the $\Sigma^1_1$-separation property in the generalized Baire space $\omega_1^{\omega_1}$. Another example, though speculative, involves lifting the above results to inner models with finitely many Woodin cardinals. We strongly believe that the proofs of the theorems above can serve as a blueprint to obtain models where the $\Sigma^1_{n+3}$-separation property holds, for arbitrary $n \in \omega$, while working over the canonical inner model with $n$ Woodin cardinals $M_n$ instead of $L$, as we do in this article. Note here, that for even $n$ this would produce models, which display a behaviour of the separation property which contradicts the one implied by $\PD$.

We finish the introduction with a short summary of the present article. In section two we briefly introduce the forcings which will be used in order to prove the two main theorems. In the third section we shall construct a mild generic extension of $L$ denoted with $W$ which is for our needs the right ground model to work with. In the third subsection of section three, we prove an auxiliary result whose purpose is to highlight several important ideas in an easier setting. Our hopes are that this way, the reader obtains a better understanding of the proofs of the later main theorems.
In section four we prove the boldface separation property and in the fifth section we prove the lightface separation property. The latter relies on several arguments from the boldface case and can not be read separately. In the sixth section we discuss some interesting and open questions.

\section{Preliminaries}

\subsection{Notation}
The notation we use will be mostly standard, we hope. We write $\forceP=(\forceP_{\alpha} \, : \, \alpha < \gamma)$ for a forcing iteration of length $\gamma$ with initial segments $\forceP_{\alpha}$. The $\alpha$-th factor of the iteration will be denoted with $\forceP(\alpha)$. Note here that we drop the dot on $\forceP(\alpha)$, even though $\forceP(\alpha)$ is in fact a $\forceP_{\alpha}$-name of a partial order. 
If $\alpha' < \alpha < \gamma$, then we write $\forceP_{\alpha' \alpha}$ to denote the intermediate forcing of $\forceP$ which happens in the interval $[\alpha',\alpha)$, i.e. $\forceP_{\alpha' \alpha}$ is such that 
$\forceP \cong \forceP_{\alpha'} \ast \forceP_{\alpha' \alpha}$.

We write $\forceP \Vdash \varphi$ whenever every condition in $\forceP$ forces $\varphi$, and make deliberate use of restricting partial orders below conditions, that is, if $p \in \forceP $ is such that $p \Vdash \varphi$, we let $\forceP':= \forceP_{\le p}:=\{ q \in \forceP \, : \, q \le p\}$ and use $\forceP'$ instead of $\forceP$. This is supposed to reduce the notational load of some definitions and arguments.
We also sometimes write $V[\forceP]\models \varphi$ to indicate that for every $\forceP$-generic filter $G$ over $V$, $V[G] \models \varphi$.

\subsection{The forcings which are used}
The forcings which we will use in the construction are all well-known. We nevertheless briefly introduce them and their main properties. 

\begin{definition}
 For a stationary $S \subset \omega_1$ the club-shooting forcing with finite conditions for $S$, denoted by $\forceP_S$ consists
 of conditions $p$ which are finite partial functions from $\omega_1$ to $S$ and for which there exists a normal function $f: \omega_1 \rightarrow \omega_1$ such that $p \subset f$. $\forceP_S$ is ordered by end-extension.
 \end{definition}
The club shooting forcing $\forceP_S$ is the paradigmatic example for an $S$-\emph{proper forcing}, where we say that $\forceP$ is $S$-proper if and only if for every condition $p \in \forceP_S$, every sufficiently large $\theta$ and every countable $M \prec H(\theta)$ such that $M \cap \omega_1 \in S$ and $p, \forceP_S \in M$, there is a $q<p$ which is $(M, \forceP_S)$-generic. 
\begin{lemma}
 The club-shooting forcing $\forceP_S$ generically adds a club through the stationary set $S \subset \omega_1$, while being $S$-proper and
 hence $\omega_1$-preserving. Moreover stationary subsets $T$ of $S$ remain stationary in the generic extension. 
\end{lemma}

We will choose a family of $S_{\beta}$'s so that we can shoot an arbitrary pattern of clubs through its elements such that this pattern can be read off
from the stationarity of the $S_{\beta}$'s in the generic extension.
For that it is crucial to recall that $S$-proper posets can be iterated with countable support and always yield an $S$-proper forcing again. This is proved exactly as in the well-known case for plain proper forcings (see \cite{Goldstern}, 3.19. for a proof).
\begin{fact}
Let $(\forceP_{\alpha} \,:\, \alpha< \gamma)$ be a countable support iteration, assume also that at every stage $\alpha$, $\forceP_{\alpha} \Vdash_{\alpha} \forceP(\alpha)$ is $S$-proper. Then the iteration is an $S$-proper notion of forcing again.
\end{fact}
Once we decide to shoot a club through a stationary, co-stationary subset of $\omega_1$, this club will belong to all $\omega_1$-preserving outer models. This hands us a robust method of coding arbitrary information into a suitably chosen sequence of sets. Let $(S_{ \alpha} \, : \, \alpha < \omega_1)$ be a sequence of stationary, co-stationary subsets of $\omega_1$ such that $\forall \alpha, \, \beta < \omega_1 (S_{\alpha} \cap S_{\beta} \in \hbox{NS}_{\omega})$, and let $S\subset \omega_1$ be stationary and such that $S \cap S_{\alpha} \in \hbox{NS}_{\omega})$. Note that we can always assume that these objects exist.
The following coding method has been used several times already (see \cite{SyVera}).
\begin{lemma}
 Let $r \in 2^{\omega_1}$ be arbitrary, and let $\forceP$ be a countable support iteration $(\forceP_{\alpha} \, : \, \alpha < \omega_1)$, inductively defined via \[\forceP(\alpha) := \forceP_{\omega_1 \backslash S_{2 \cdot \alpha}} \text{ if } r(\alpha)=1 \] and
 \[\forceP({\alpha}) := \forceP_{\omega_1 \backslash S_{(2 \cdot\alpha) +1}} \text{ if } r(\alpha)=0.\]
 Then in the resulting generic extension $V[\forceP]$, we have that $\forall \alpha < \omega_1:$ \[ r(\alpha)=1 \text{ if and only if }
 S_{2 \cdot \alpha} \text{  is nonstationary, }\]  and \[ r_{\alpha}=0 \text{ iff } S_{(2 \cdot \alpha)+1} \text{ is nonstationary.} \]
\end{lemma}

\begin{proof}
Note first that the iteration will be $S$-proper, hence $\omega_1$-preserving. Assume  that $r(\alpha)=1$ in $V[{\forceP}]$. Then by definition of the iteration
 we must have shot a club through the complement of $S_{\alpha}$, thus it is nonstationary
 in $V[{\forceP}]$. 
 
 On the other hand, if $S_{2 \cdot \alpha}$ is nonstationary in $V[{\forceP}]$, then
 as for $\beta \ne 2 \cdot \alpha$, every forcing of the form $\forceP_{S_{\beta}}$ is $S_{2 \cdot \alpha}$-proper,
 we can iterate with countable support and preserve $S_{2 \cdot \alpha}$-properness, thus the stationarity of $S_{2 \cdot \alpha}$.
 So if $S_{2 \cdot \alpha}$ is nonstationary in $V[{\forceP}]$, we must have used $\forceP_{S_{2 \cdot \alpha}}$ in the iteration,
 so $r(\alpha)= 1$.
 
\end{proof}
The second forcing we use is the almost disjoint coding forcing due to R. Jensen and R. Solovay. We will identify subsets of $\omega$ with their characteristic function and will use the word reals for elements of $2^{\omega}$ and subsets of $\omega$ respectively.
Let $D=\{d_{\alpha} \, : \, \alpha < \aleph_1 \}$ be a family of almost disjoint subsets of $\omega$, i.e. a family such that if $d, d' \in D$ then 
$d \cap d'$ is finite. Let $X\subset  \kappa$ for $\kappa \le 2^{\aleph_0}$ be a set of ordinals. Then there 
is a ccc forcing, the almost disjoint coding $\mathbb{A}_D(X)$ which adds 
a new real $x$ which codes $X$ relative to the family $D$ in the following way
$$\alpha \in X \text{ if and only if } x \cap d_{\alpha} \text{ is finite.}$$
\begin{definition}
 The almost disjoint coding $\mathbb{A}_D(X)$ relative to an almost disjoint family $D$ consists of
 conditions $(r, R) \in \omega^{<\omega} \times D^{<\omega}$ and
 $(s,S) < (r,R)$ holds if and only if
 \begin{enumerate}
  \item $r \subset s$ and $R \subset S$.
  \item If $\alpha \in X$ and $d_{\alpha} \in R$ then $r \cap d_{\alpha} = s \cap d_{\alpha}$.
 \end{enumerate}
\end{definition}
For the rest of this paper we let $D \in L$ be the definable almost disjoint family of reals one obtains when recursively adding the $<_L$-least real to the family which is almost disjoint from all the previously picked reals. 
Whenever we use almost disjoint coding forcing, we assume that we code relative to this fixed almost disjoint family $D$.

The last two forcings we briefly discuss are Jech's forcing for adding a Suslin tree with countable conditions and, given a Suslin tree $T$, the associated forcing which adds a cofinal branch through $T$. 
Recall that a set theoretic tree $(T, <)$ is a Suslin tree if it is a normal tree of height $\omega_1$
and has no uncountable antichain. As a result, forcing with a Suslin tree $S$, where conditions are just nodes in $S$, and which we always denote with $S$ again, is a ccc forcing of size $\aleph_1$. 
Jech's forcing to generically add a Suslin tree is defined as follows.

\begin{definition}
 Let $\forceP_J$ be the forcing whose conditions are
 countable, normal trees ordered by end-extension, i.e. $T_1 < T_2$ if and only
 if $\exists \alpha < \text{height}(T_1) \, T_2= \{ t \upharpoonright \alpha \, : \, t \in T_1 \}.$
\end{definition}
It is wellknown that $\forceP_J$ is $\sigma$-closed and
adds a Suslin tree. In fact more is true, the generically added tree $T$ has 
the additional property that for any Suslin tree $S$ in the ground model
$S \times T$ will be a Suslin tree in $V[G]$. This can be used to obtain a robust coding method (see also \cite{Ho} for more applications)

\begin{lemma}\label{oneSuslintreepreservation}
 Let $V$ be a universe and let $S \in V$ be a Suslin tree. If $\forceP_J$ is 
 Jech's forcing for adding a Suslin tree, $g \subset \forceP_J$ be a generic filter and if $T=\bigcup g$ is the generic tree
 then $$V[g][T] \models  S \text{ is Suslin.}$$
\end{lemma}

\begin{proof}
Let $\dot{T}$ be the $\forceP_J$-name for the generic Suslin tree. We claim that $\forceP_J \ast \dot{T}$ has a dense subset which is $\sigma$-closed. As $\sigma$-closed forcings will always preserve ground model Suslin trees, this is sufficient. To see why the claim is true consider the following set:
$$\{ (p, \check{q}) \, : \, p \in \forceP_J \land height(p)= \alpha+1  \land  \check{q} \text{ is a node of $p$ of level } \alpha \}.$$
It is easy to check that this set is dense and $\sigma$-closed in $\forceP_J \ast \dot{T}$.

\end{proof}

A similar observation shows that we can add an $\omega_1$-sequence of
such Suslin trees with a countably supported iteration.

\begin{lemma}\label{ManySuslinTrees}
 Let $S$ be a Suslin tree in $V$ and let $\forceP$ be a countably supported
 product of length $\omega_1$ of forcings $\forceP_J$. Then in the generic extension
 $V[G]$ there is an $\omega_1$-sequence of Suslin trees $\vec{T}=(T_{\alpha} \, : \, \alpha \in \omega_1)$ such
that for any finite $e \subset \omega$
the tree $S \times \prod_{i \in e} T_i$ will be a Suslin tree in $V[\vec{T}]$.
\end{lemma}

These sequences of Suslin trees will be used for coding in our proof and get a name.
\begin{definition}
 Let $\vec{T} = (T_{\alpha} \, : \, \alpha < \kappa)$ be a sequence of Suslin trees. We say that the sequence is an 
 independent family of Suslin trees if for every finite set $e= \{e_0, e_1,...,e_n\} \subset \kappa$ the product $T_{e_0} \times T_{e_1} \times \cdot \cdot \cdot \times T_{e_n}$ 
 is a Suslin tree again.
\end{definition}
The upshot of being an independent sequence is that we can pick our favourite subset of indices and decide to shoot a branch through every tree whose index belongs to the set, while guaranteeing that no other Suslin tree from the sequence is destroyed. The following fact can be easily seen via induction on $\kappa$.
\begin{fact}
Let $\vec{T} = (T_{\alpha} \, : \, \alpha < \kappa)$ be independent and let $I \subset \kappa$ be arbitrary.
If we form the finitely supported product of forcings $\forceP:=\prod_{\alpha \in I} T_{\alpha}$, then for every $\beta \notin I$,
$V[\forceP] \models ``T_{\beta}$ is a Suslin tree$"$.
\end{fact}
Thus independent Suslin trees are suitable to encode information, as soon as we can make the independent sequence definable. 
\section{A first step towards the proof of the boldface separation property}

\subsection{The ground model $W$ of the iteration}
We have to first create a suitable ground model $W$ over which the actual iteration will take place. $W$ will be a generic extension of $L$, satisfying $\CH$ and, as stated already earlier, has the property that it contains two $\omega_1$-sequences $\vec{S}=\vec{S^1} \cup \vec{S^2}$ of mutually independent Suslin trees. 
The goal is to add the trees generically, and in a second forcing, use an $L$-definable sequence of stationary subsets of $\omega_1$ to code up the trees. The resulting universe will have the feature that any further outer universe, which preserves stationary subsets, can decode the information written into the $L$-stationary subsets in a $\Sigma_1(\omega_1)$-definable way, and hence has access to the sequence of independent Suslin trees $\vec{S}$. This property can be used to create two $\Sigma_1(\omega_1)$-predicates which are empty first and which can be filled with arbitrary reals $x$, using $\aleph_1$-sized forcings with the countable chain condition. These forcing have the crucial feature that they will be independent of the ground model they live in, a feature we will exploit heavily later on.

We start with G\"odels constructible universe $L$ as our 
ground model.
Next we fix an appropriate sequence of stationary subsets of $\omega_1$.
Recall that $\diamondsuit$ holds in our ground model $L$, i.e. there is a 
$\Sigma_1$-definable sequence $(a_{\alpha} \, : \, \alpha < \omega_1)$ of countable subsets of $\omega_1$
such that any set $A \subset \omega_1$ is guessed stationarily often by the $a_{\alpha}$'s, i.e.
$\{ \alpha < \omega_1 \, : \, a_{\alpha}= A \cap \alpha \}$ is a stationary subset of $\omega_1$. 
The $\diamondsuit$-sequence can be used to produce an easily definable sequence of stationary subsets: using a definable bijection between $\omega_1$ and $\omega_1 \cdot \omega_1$, we list the reals in $L$ in an $\omega_1 \cdot \omega_1$ sequence $(r_{\beta} \, : \, \beta < \omega_1 \cdot \omega_1)$ and define for every $\beta < \omega_1 \cdot \omega_1$
a stationary set in the following way:
$$R_{\beta} := \{ \alpha < \omega_1  \, : \, a_{\alpha}= r_{\beta}   \}.$$
and let $\vec{R}= (R_{\beta} \, : \, \beta < \omega_1 \cdot \omega_1)$ denote the sequence.

We proceed with adding an $\omega_1$-sequence of Suslin trees with a countably supported product of Jech's Forcing $ \forceP_J$. We let 
\[\forceR := \prod_{\beta \in \omega_1} \forceP_J \] using countable support. This is a $\sigma$-closed, hence proper notion of forcing. We denote the generic filter of $\forceR$ with $\vec{S}=(S_{\alpha} \, : \, \alpha < \omega_1)$ and note that whenever $I \subset \omega_1$ is a set of indices then for every $j \notin I$, the Suslin tree $S_j$ will remain a Suslin tree in the universe $L[\vec{S}][g]$, where $g \subset \prod_{i \in I} S_i$ denotes the generic filter for the forcing with the finitely supported product of the trees $S_i$, $i \in I$ (see \cite{Ho} for a proof of this fact). We fix a definable bijection between $[\omega_1]^{\omega}$ and $\omega_1$ and identify the trees in $(S_{\alpha }\, : \, \alpha < \omega_1)$ with their images under this bijection, so the trees will always be subsets of $\omega_1$ from now on. 

In a second step, we destroy each element of $\vec{S}$, via adding generically branches. That is, we let the second forcing
\[ \forceR':= \prod_{\alpha < \omega_1}  S_{\alpha}, \]
using countable support. Note that by the argument from the proof of Lemma \ref{oneSuslintreepreservation}, $\forceR \ast \forceR'$ has a dense subset which is $\sigma$-closed, hence the $L[\forceR \ast \forceR']$ is a $\sigma$-closed generic extension of $L$.

In a third step we code the trees from $\vec{S}$ into the  sequence of $L$-stationary subsets $\vec{R}$ we produced earlier, using club shooting forcing. It is important to note, that the forcing we are about to define does preserve Suslin trees, a fact we will show later.
The forcing used in the second third will be denoted by $\mathbb{S}$. Fix $\alpha< \omega_1$ and consider the $\omega_1$ tree $S_{\alpha} \subset \omega_1$. We let $\forceR_{\alpha}$ be the countable support product which codes the characteristic function of $S_{\alpha}$ into the $\alpha$-th $\omega_1$-block of the $R_{\beta}$'s. 

$$\forceR_{\alpha} = \prod_{\gamma \in S_{\alpha}} \forceP_{\omega_1 \backslash R_{\omega_1 \cdot \alpha + 2 \cdot \gamma}} \times \prod_{\gamma \notin S_{\alpha}} \forceP_{\omega_1 \backslash R_{\omega_1 \cdot \alpha  + 2 \cdot \gamma +1}} $$ 
Recall that for a stationary, co-stationary $R \subset \omega_1$, $\forceP_{R}$ denotes the club shooting forcing which shoots a club through $R$, thus $\forceR_{\alpha}$ codes up the tree $S_{\alpha}$ via writing the 0,1-pattern of the characteristic function of $S_{\alpha}$ into the $\alpha$-th $\omega_1$-block of $\vec{R}$.

If we let $R$ be some stationary subset of $\omega_1$ which is disjoint from all the $R_{\alpha}$'s, whose existence is guaranteed by $\diamondsuit$, then it is obvious that for every $\alpha < \omega_1$, $\forceR_{\alpha}$ is an $R$-proper forcing which additionally is $\omega$-distributive.  Then we let $\mathbb{S}$ be the countably supported iteration, $$\mathbb{S}:=\bigstar_{\alpha< \omega_1} \forceR_{\alpha}$$ which is again $R$-proper and $\omega$-distributive.
This way we can turn the generically added sequence of $\omega_1$-trees $\vec{S}$ into a definable sequence of $\omega_1$-trees.
Indeed, if we work in $L[\vec{S}\ast G]$, where $\vec{S} \ast G$ is $\forceR \ast \mathbb{S}$-generic over $L$, then 
\begin{align*}
\forall \alpha, \gamma < \omega_1 (&\gamma \in S_{\alpha} \Leftrightarrow R_{\omega_1 \cdot \alpha + 2 \cdot \gamma} \text{ is not stationary and} \\ &
\gamma \notin S_{\alpha} \Leftrightarrow  R_{\omega_1 \cdot \alpha + 2 \cdot \gamma +1} \text{ is not stationary})
\end{align*}
Note here that the above formula can be written in a $\Sigma_1(\omega_1)$-way, as  it reflects down to $\aleph_1$-sized, transitive models of $\ZFP$ which contain a club through exactly one element of every pair $\{(R_{\alpha}, R_{\alpha+1}) \, : \, \alpha < \omega_1\}$.
Finally we partition $\vec{S}$ into its even and its odd members and let 
\[ \vec{S^1}:= \{ S_{\alpha} \in \vec{S} \, : \, \alpha \text{ is even } \} \]
and 
\[ \vec{S^2}:= \{ S_{\beta} \in \vec{S} \, : \, \beta \text { is odd} \} \]
Again, both sequences $\vec{S^1}$ and $\vec{S^2}$ are $\Sigma^1(\omega_1)$-definable in $W$ and all stationary set preserving outer models of $W$.

Our goal is to use $\vec{S^1}$ and $\vec{S^2}$ for coding again. For this it is essential, that both sequences remain independent in the inner model $L[\forceR][\mathbb{S}]$,  after forcing with $\mathbb{S}$. 
The following  line of reasoning is similar to \cite{Ho}.
Recall that for a forcing $\forceP$ and $M \prec H(\theta)$, a condition $q \in \forceP$ is $(M,\forceP)$-generic iff for every maximal antichain $A \subset \forceP$, $A \in M$, it is true that $ A \cap M$ is predense below $q$.
The key fact is the following (see \cite{Miyamoto2} for the case where $\forceP$ is proper)
\begin{lemma}\label{preservation of Suslin trees}
 Let $T$ be a Suslin tree, $S \subset \omega_1$ stationary and $\forceP$ an $S$-proper
 poset. Let $\theta$ be a sufficiently large cardinal.
 Then the following are equivalent:
 \begin{enumerate}
  \item $\Vdash_{\forceP} T$ is Suslin
 
  \item if $M \prec H_{\theta}$ is countable, $\eta = M \cap \omega_1 \in S$, and $\forceP$ and $T$ are in $M$,
  further if $p \in \forceP \cap M$, then there is a condition $q<p$ such that 
  for every condition $t \in T_{\eta}$, 
  $(q,t)$ is $(M, \forceP \times T)$-generic.
 \end{enumerate}

\end{lemma}

\begin{proof}
For the direction from left to right note first that $\Vdash_{\forceP} T$ is Suslin implies $\Vdash_{\forceP} T$ is ccc, and in particular it is true that for any countable elementary submodel $N[\dot{G}_{\forceP}] \prec H(\theta)^{V[\dot{G}_{\forceP}]}$,  $\Vdash_{\forceP} \forall t \in T (t$ is $(N[\dot{G}_{\forceP}],T)$-generic). Now if $M \prec H(\theta)$ and $M \cap \omega_1 = \eta \in S$ and $\forceP,T \in M$ and $p \in \forceP \cap M$ then there is a $q<p$ such that $q$ is $(M,\forceP)$-generic. So $q \Vdash \forall t \in T (t$ is $(M[\dot{G}_{\forceP}], T)$-generic, and this in particular implies that $(q,t)$ is $(M, \forceP \times T)$-generic for all $t \in T_{\eta}$. 

For the direction from right to left assume that $\Vdash \dot{A} \subset T$ is a maximal antichain. Let $B=\{(x,s) \in \forceP \times T \, : \, x \Vdash_{\forceP} \check{s} \in \dot{A} \}$, then $B$ is a predense subset in $\forceP \times T$. Let $\theta$ be a sufficiently large regular cardinal and let $M \prec H(\theta)$ be countable such that $M \cap \omega_1=\eta \in S$ and $\forceP, B,p,T \in M$. By our assumption there is a $q <_{\forceP} p$  such that $\forall t \in T_{\eta} ((q,t)$ is $(M, \forceP \times T)$-generic). So $B \cap M$ is predense below $(q,t)$ for every $t \in T_{\eta}$, which yields that $q \Vdash_{\forceP} \forall t \in T_{\eta} \exists s<_{T} t(s \in \dot{A})$ and hence $q \Vdash \dot{A} \subset T \upharpoonright \eta$, so $\Vdash_{\forceP} T$ is Suslin.
\end{proof}
In a similar way, one can show that Theorem 1.3 of \cite{Miyamoto2} holds true if we replace proper by $S$-proper for $S \subset \omega_1$ a stationary subset.
\begin{theorem}
Let $(\forceP_{\alpha})_{\alpha < \eta}$ be a countable support iteration of length $\eta$, let $S \subset \omega_1$ be stationary and suppose that for every $\alpha < \eta$, for the $\alpha$-th factor of the iteration $\dot{\forceP}(\alpha)$ it holds that $\Vdash_{\alpha}  ``\dot{\forceP}(\alpha)$ is $S$-proper and 
preserves every Suslin tree.$"$ Then $\forceP_{\eta}$ is $S$-proper and preserves every Suslin tree.
\end{theorem}
So in order to argue that our forcing $\mathbb{S}$ preserves Suslin trees if used over $L[\forceR]$, it is sufficient to show that every factor preserves Suslin trees.
This is indeed the case.
\begin{lemma}
Let $S \subset \omega_1$ be stationary, co-stationary, then the club shooting forcing $\forceP_S$ preserves Suslin trees.
\end{lemma}

\begin{proof}
Because of Lemma \ref{preservation of Suslin trees}, it is enough to show that for any regular and sufficiently large
 $\theta$, every $M \prec H_{\theta}$ with $M \cap \omega_1 = \eta \in S$, and every
 $p \in \forceP_S \cap M$ there is a $q<p$ such that for every
 $t \in T_{\eta}$, $(q,t)$ is $(M,(\forceP_S \times T))$-generic.
 Note first that as $T$ is Suslin, every node $t \in T_{\eta}$ is an
 $(M,T)$-generic condition. Further, as forcing with a Suslin tree
 is $\omega$-distributive, $M[t]$ has the same $M[t]$-countable sets as $M$.
 It is not hard to see that if $M\prec H(\theta)$ is such
 that $M \cap \omega_1 \in S$ then an $\omega$-length descending sequence
 of $\forceP_S$-conditions in $M$ whose domains converge to $M \cap \omega_1$
 has a lower bound as $M \cap \omega_1 \in S$.
 
 We construct an $\omega$-sequence of elements of $\forceP_S$ which has a lower bound
 which will be the desired condition. 
 We list the nodes on $T_{\eta}$, $(t_i \, : \, i \in \omega)$ and
 consider the according generic extensions $M[t_i]$.
 In every $M[t_i]$ we list the $\forceP_S$-dense subsets of $M[t_i]$,
 $(D^{t_i}_n \, : \, n \in \omega)$ and write 
 the so listed dense subsets of $M[t_i]$ as an $\omega \times \omega$-matrix and enumerate
 this matrix in an $\omega$-length sequence of dense sets $(D_i \, : \, i \in \omega)$.
 If $p=p_0 \in \forceP_S \cap M$ is arbitrary we can find, using the fact that $\forall i \, (\forceP_S \cap M[t_i] = M \cap \forceP_S$), an $\omega$-length, descending
 sequence of conditions below $p_0$ in $\forceP_S \cap M$, $(p_i \, : \, i \in \omega)$
 such that $p_{i+1} \in M \cap \forceP_S$ is in $D_i$.
 We can also demand that the domain of the conditions $p_i$ converge to $M \cap \omega_1$.
 Then the $(p_i)$'s have a lower bound $p_{\omega} \in \forceP_S$ and $(t, p_{\omega})$ is an
 $(M, T \times \forceP_S)$-generic conditions for every $t \in T_{\eta}$ as any $t \in T_{\eta}$ is $(M,T)$-generic
 and every such $t$ forces that $p_{\omega}$ is $(M[T], \forceP_S)$-generic; moreover $p_{\omega} < p$ as 
 desired.
 \end{proof}
 
Putting things together we obtain:

\begin{theorem}
The forcing $\mathbb{S}$, defined above preserves Suslin trees.
\end{theorem}

Let us set $W:= L[\forceR \ast (\forceR' \times \mathbb{S}) ]$ which will serve as our ground model for a second iteration of length $\omega_1$. Note that $W$ satisfies that it is an $\omega$-distributive generic extension of $L$.

We end with a straightforward lemma which is used later in coding arguments.

\begin{lemma}\label{a.d.coding preserves Suslin trees}
 Let $T$ be a Suslin tree and let $\mathbb{A}_F(X)$ be the almost disjoint coding which codes
 a subset $X$ of $\omega_1$ into a real with the help of an almost disjoint family
 of reals of size $\aleph_1$. Then $$\Vdash_{\mathbb{A}_{F}(X)} T \text{ is Suslin }$$
 holds.
\end{lemma}
\begin{proof}
 This is clear as $\mathbb{A}_{F}(X)$ has the Knaster property, thus the product $\mathbb{A}_{F}(X) \times T$ is ccc and $T$ must be Suslin in $V^{\mathbb{A}_{F}(X)}$. 
\end{proof}

\subsection{Coding reals into Suslin trees}
We introduced the model $W$ for one specific purpose: the possibility to code up reals into the sequence of definable Suslin trees $\vec{S^1}$ or $\vec{S^2}$ using a method which is not sensitive to its ground model.

For the following, we let $W$ be our ground model, though the definitions will work, and will be used for suitable outer models of $W$ as well. We will encounter this situation as ultimately we will iterate the coding forcings we are about to define.
Let $x \in W$  be an arbitrary real, ;et $m,k \in \omega$, let $(x,m,k)$ denote the real which codes the triple consisting of $x,m$ and $k$ in some fixed recursive way, and let $i \in \{ 1,2\}$. Then we shall define the forcing $\hbox{Code}(x,i)$, which codes the real $x$ into $\aleph_1$-many $\omega$-blocks of $\vec{S^i}$ as a two step iteration:

\[ \hbox{Code}((x,m,k),i):= \mathbb{C} (\omega_1)^L \ast \dot{\mathbb{A}} (\dot{Y}_{x,i}) \]
where the first factor is ordinary $\omega_1$-Cohen forcing, but defined in $L$, and the second factor codes a specific subset of $\omega_1$ denoted with $Y_{(x,m,k),i}$ into a real using almost disjoint coding forcing relative to the canonical, constructible almost disjoint family of reals $D$.
We emphasize, that in iterations of coding forcings, we still fall back to force with $(\mathbb{C} (\omega_1))^L$ as our first factor, that is we never use the $\omega_1$-Cohen forcing of the current universe. Thus, iterating the coding forcings is in fact a hybrid of a product (namely the coordinates where we use $(\mathbb{C}(\omega_1))^L$)  and a finites support iteration (the coordinates where we use the almost disjoint coding forcing). We shall discuss this later in more detail.

We let $g \subset \omega_1$ be a $\mathbb{C} (\omega_1)^L$-generic filter over $W$, and let $\rho: [\omega_1]^{\omega} \rightarrow \omega_1$ be some canonically definable, constructible bijection between
these two sets. We use $\rho$ and $g$ to define the set $h \subset \omega_1$, which eventually shall be the set of indices of $\omega$-blocks of $\vec{S}^i$, where we code up the characteristic function of the real ($(x,m,k)$. Let $h:= \{\rho( g \cap \alpha) \,: \, \alpha < \omega_1 \}$ and let $X \subset \omega_1$ be the $<$-least set  (in some previously fixed well-order of $H(\omega_2)^{W[g]}$ which codes the follwing objects:
\begin{itemize}
\item The $<$-least set of $\omega_1$-branches in $W$ through elments of $\vec{S}$ which code $(x,m,k)$ at $\omega$-blocks which start at values in $h$, that is  we collect $\{ b_{\beta} \subset S_{\beta} \, : \, \beta= \omega \gamma + 2n, \gamma \in h \land n \in \omega \land n \notin (x,m,k) \}$ and  $\{ b_{\beta} \subset S_{\beta} \, : \, \beta= \omega \gamma + 2n+1, \gamma \in h \land n \in \omega \land n \in (x,m,k) \}$.
\item The $<$-least set of $\omega_1 \cdot \omega \cdot \omega_1$-many club subsets through $\vec{R}$, our $\Sigma_1 (\omega_1)$-definable sequence of $L$-stationary subsets of $\omega_1$ from the last section, which are necessary to compute every tree $S_{\beta} \in \vec{S}$ which shows up in the above item, using the $\Sigma_1 (\omega_1)$-formula from the previous section before Lemma 2.10.
\end{itemize}

Note that, when working in $L[X]$ and if $\gamma \in h$ then
 we can read off $(x,m,k)$ via looking at the $\omega$-block of $\vec{S^i}$-trees starting at $\gamma$ and determine which tree has an $\omega_1$-branch in $L[X]$:
\begin{itemize}
 \item[$(\ast)$]  $n \in (x,m,k)$ if and only if $S^i_{\omega \cdot \gamma +2n+1}$ has an $\omega_1$-branch, and $n \notin (x,m,k)$ if and only if $S^i_{\omega \cdot \gamma +2n}$ has an $\omega_1$-branch.
\end{itemize}
Note that $(\ast)$ is actually a formula $(\ast) ((x,y,m) ,\gamma)$ with two parameters $(x,y,m)$ and $\gamma$ but we will suppress this, as the parameters usually are clear from the context.
Indeed if $n \notin (x,m,k)$ then we added a branch through $S^i_{\omega \cdot \gamma+ 2n}$. If on the other hand $S^i_{\omega \cdot\gamma +2n}$ is 
Suslin in $L[X]$ then we must have added an $\omega_1$-branch through $S^i_{\omega \cdot \gamma +2n+1}$ as we always add an $\omega_1$-branch through either $S^i_{\omega \cdot \gamma +2n+1}$ or $S^i_{\omega \cdot \gamma +2n}$ and adding branches through some $S^i_{\alpha}$'s  will not affect that some $S^i_{\beta}$ is Suslin in $L[X]$, as $\vec{S}$ is independent.

We note that we can apply an argument resembling David's trick in this situation. We rewrite the information of $X \subset \omega_1$ as a subset $Y \subset \omega_1$ using the following line of reasoning.
It is clear that any transitive, $\aleph_1$-sized model $M$ of $\ZFP$ which contains $X$ will be able to correctly decode out of $X$ all the information.
Consequentially, if we code the model $(M,\in)$ which contains $X$ as a set $X_M \subset \omega_1$, then for any uncountable $\beta$ such that $L_{\beta}[X_M] \models \ZFP$ and $X_M \in L_{\beta}[X_M]$:
\[L_{\beta}[X_M] \models \text{\ldq The model decoded out of }X_M \text{ satisfies $(\ast)$ for every $\gamma \in h \subset \omega_1$\rdq.} \]
In particular there will be an $\aleph_1$-sized ordinal $\beta$ as above and we can fix a club $C \subset \omega_1$ and a sequence $(M_{\alpha} \, : \, \alpha \in C)$ of countable elementary submodels such that
\[\forall \alpha \in C (M_{\alpha} \prec L_{\beta}[X_M] \land M_{\alpha} \cap \omega_1 = \alpha)\]
Now let the set $Y\subset \omega_1$ code the pair $(C, X_M)$ such that the odd entries of $Y$ should code $X_M$ and if $Y_0:=E(Y)$ where the latter is the set of even entries of $Y$ and $\{c_{\alpha} \, : \, \alpha < \omega_1\}$ is the enumeration of $C$ then
\begin{enumerate}
\item $E(Y) \cap \omega$ codes a well-ordering of type $c_0$.
\item $E(Y) \cap [\omega, c_0) = \emptyset$.
\item For all $\beta$, $E(Y) \cap [c_{\beta}, c_{\beta} + \omega)$ codes a well-ordering of type $c_{\beta+1}$.
\item For all $\beta$, $E(Y) \cap [c_{\beta}+\omega, c_{\beta+1})= \emptyset$.
\end{enumerate}
We obtain
\begin{itemize}
\item[$({\ast}{\ast})$] For any countable transitive model $M$ of $\ZFP$ such that $\omega_1^M=(\omega_1^L)^M$ and $ Y \cap \omega_1^M \in M$, $M$ can construct its version of the universe $L[Y \cap \omega_1^M]$, and the latter will see that there is an $\aleph_1^M$-sized transitive model $N \in L[Y \cap \omega_1^M]$ which models $(\ast)$ for $(x,m,k)$ and every $\gamma \in h \subset \omega_1^M$.
\end{itemize}
Thus we have a local version of the property $(\ast)$.

In the next step $\dot{\mathbb{A}} (\dot{Y})$, working in $W[g]$, for $g\subset \mathbb{C} (\omega_1)$ generic over $W$, we use almost disjoint forcing $\mathbb{A}_D(Y)$ relative to the $<_L$-least almost disjoint family of reals $D \in  L $ to code the set $Y$ into one real $r$. This forcing is well-known, has the ccc and its definition only depends on the subset of $\omega_1$ we code, thus the almost disjoint coding forcing  $\mathbb{A}_D(Y)$ will be independent of the surrounding universe in which we define it, as long as it has the right $\omega_1$ and contains the set $Y$.

We finally obtained a real $r$ such that
\begin{itemize}
\item[$({\ast}{\ast}{\ast})$] For any countable, transitive model $M$ of $\ZFP$ such that $\omega_1^M=(\omega_1^L)^M$ and $ r  \in M$, $M$ can construct its version of $L[r]$ which in turn thinks that there is a transitive $\ZFP$-model $N$ of size $\aleph_1^M$  such that $N$ believes $(\ast)$ for $(x,m,k)$ and every $\gamma \in h$.
\end{itemize}

Note that the above is a $\Pi^1_2(r)$-statement.
We say in this situation that the real $(x,m,k)$\emph{ is written into $\vec{S}^i$}, or that $(x,m,k)$ \emph{is coded into} $\vec{S^i}$.
If $(x,m,k)$ is coded into $\vec{S}^i$ and $r$ is a real witnessing this, then the set $h$ which is equal to $\{ \gamma < \omega_1 \, ;\, \gamma \text{ is a starting point for an $\omega$-block where $(\ast)$ }$ for $(x,y,m)$ holds$\}$ is dubbed (following \cite{FS}) the coding area of $(x,m,k)$ with respect to $r$.

We want to iterate these coding forcings. As the first factor of a coding forcing will always be $(\mathbb{C}(\omega_1))^L$, an iteration of the coding forcing is in fact a hybrid of a (countably supported) product (namely the coordinates where we use $(\mathbb{C}(\omega_1))^L$)  and an actual finite support iteration (the coordinates where we use almost disjoint coding forcing).

\begin{definition}
A mixed support iteration $\forceP=(\forceP_{\beta}\,:\, {\beta< \alpha})$ is called legal if $\alpha < \omega_1$ and there exists a bookkeeping function $F: \alpha \rightarrow H(\omega_2)^2$ such that 
 $\forceP$ is defined inductively using $F$ as follows:
 \begin{itemize}
 \item If $F(0)=(x,i)$, where $x$ is a real, $i\in \{1,2\}$, then $\forceP_0= \hbox{Code}({x,i} )$. Otherwise $\forceP_0$ is the trivial forcing.
 \item If $\beta>0$ and $\forceP_{\beta}$ is defined, $G_{\beta} \subset \forceP_{\beta}$ is a generic filter over $W$, $F(\beta)=(\dot{x}, i)$, where $\dot{x}$ is a $\forceP_{\beta}$-name of a real, $i \in \{1,2\}$ and $\dot{x}^{G_{\beta}}=x$ then, working in $W[G_{\beta}]$ we let
 $\forceP(\beta):= \hbox{Code} ({x,i} )$, that is we code $x$ into the $\vec{S}^i$, using our coding forcing. We shall use full (i.e. countable) support on the $(\mathbb{C}(\omega_1))^L$-coordinates and finite support on the coordinates where we use almost disjoint coding forcing.
 \end{itemize}
 
\end{definition}
 Informally speaking, a legal  forcing just decides to code the reals which the bookkeeping $F$ provides into either $\vec{S^1}$ or $\vec{S^2}$. Note further that the notion of legal  can be defined in exactly the same way over any $W[G]$, where $G$ is a $\forceP$-generic filter over $W$ for an legal  forcing. Finally note that instead of creating $\omega$-blocks of Suslin trees using $\mathbb{C} (\omega_1)^L$ where we code the branches every single time we code a real, we could have also defined an altered ground model $W'$ as $W[g]$, where $g \subset \prod \mathbb{C} (\omega_1)$ is generic for the countably supported product of $\aleph_1$-many copies of $\omega_1$-Cohen forcing, and then worked over $W'$ using exclusively almost disjoint coding forcings which pick first one coordinate $g_{\alpha}$, $\alpha< \omega_1$ of $g$ in an injective way, and then code the $\aleph_1$-many branches along $g_{\alpha}$ using almost disjoint coding forcings as described above. The difference between these approaches is only of symbolic nature, we opted for the one we chose because of a slightly neater presentation. 
 
 We obtain the following first properties of legal  forcings:
 \begin{lemma}

\begin{enumerate}
\item If $\forceP=(\forceP(\beta) \, : \, \beta < \delta) \in W$ is legal  then for every $\beta < \delta$, $\forceP_{\beta} \Vdash| \forceP(\beta)|= \aleph_1$, thus every factor of $\forceP$ is forced to have size $\aleph_1$.
\item Every legal  forcing over $W$ preserves $\aleph_1$ and $\CH$.
\item The product of two legal  forcings is legal  again.
\end{enumerate}
\end{lemma}
\begin{proof}
The first assertion follows immediately from the definition.

To see the second item we exploit some symmetry.
Indeed, every legal  $\forceP = \bigstar_{\beta < \delta} P(\beta)= \bigstar_{\beta < \delta} ( ((\mathbb{C} (\omega_1))^L \ast \dot{\mathbb{A}} (\dot{Y_{\beta} }) )$ can be rewritten as \[(\prod_{\beta < \delta}  (\mathbb{C} (\omega_1))^L  )\ast \bigstar_{\beta < \delta}  \dot{\mathbb{A}}_D (\dot{Y}_{\beta} )\] (again with mixed support). The latter representation is easily seen to be of the form $\forceP \ast \bigstar_{\beta < \delta}  \dot{\mathbb{A}}_D(\dot{Y}_{\beta} )$, where $\forceP$ is $\sigma$-closed and the second part is a finite support iteration of ccc forcings, hence $\aleph_1$ is preserved. That $\CH$ holds is standard.

To see that the third item is true, we recall that the definition of $\hbox{Code} _{x,i}$ is independent of the surrounding universe as long as it contains the real $x$, thus we see that a two step iteration $\forceP_1 \ast \forceP_2$ of two legal  $\forceP_1, \forceP_2 \in W$ is in fact a product.  As the iteration of two legal  forcings (in fact the iteration of countably many legal  forcings) is legal  as well, the proof is done.
\end{proof}

The second assertion of the last lemma immediately gives us the following:
\begin{corollary}
Let $\forceP= (\forceP(\beta) \, : \, \beta < \delta) \in W$ be an legal  forcing over $W$. Then $W[\forceP] \models \CH$. Further, if $\forceP= (\forceP(\alpha) \, : \, \alpha < \omega_1) \in W$ is an $\omega_1$-length iteration such that each initial segment of the iteration is legal  over $W$, then $W[\forceP] \models \CH$.

\end{corollary}

In an iteration of coding forcing we do not add any unwanted or accidental solutions to our $\Sigma^1_3$ predicate give by $({\ast} {\ast} {\ast})$, which we shall show now.

The set of triples of (names of) reals which are enumerated by the bookkeeping function $F \in W$ which comes along with an legal $\forceP = (\forceP(\beta) \, : \, \beta < \delta)$, we call the set of reals coded by $\forceP$. That is, if \[ \forceP(\beta)= (\mathbb{C}(\omega_1))^L \ast \dot{\mathbb{A}}_D (\dot{Y}_{(\dot{x}_{\beta}, \dot{y}_{\beta}, \dot{m}_{\beta} ) } ) \] and $G \subset \forceP$ is a generic filter and if we let for every $\beta < \delta$,
$ \dot{x}_{\beta}^G =:x_{\beta}$, $\dot{y}_{\beta}^G =:y_{\beta}$, $\dot{m}_{\beta}^G =:m_{\beta}$,  then
$\{ (x_{\beta},y_{\beta},m_{\beta} ) \, : \, \beta < \alpha \}$ is the set of reals coded by $\forceP$ and $G$ (though we will suppress the $G$).

\begin{lemma}
If $\forceP \in W$ is legal, $\forceP=(\forceP_{\beta} \, : \, \beta < \delta)$, $G \subset \forceP$ is generic over $W$ and $\{ (x_{\beta},y_{\beta},m_{\beta} ) \, : \, \beta < \delta\}$ is the set of (triples of) reals which is coded as we use $\forceP$. Let 
\[A:= \{ (x,m,k) \in W[G] \, : \, \exists r ( ({\ast}{\ast}{\ast} )\text{ holds for $r$ and $(x,m,k)$} \}.  \] Then
in $W[G]$, the set of reals which belong to $A$ is exactly 
$\{ (x_{\beta},y_{\beta},m_{\beta} ) \, : \, \beta < \delta\}$, that is, we do not code any unwanted information accidentally.
\end{lemma}
\begin{proof}
Let $G$ be $\forceP$ generic over $W$. Let $g= (g_{\beta} \, : \, {\beta} < \delta)$ be the set of the $\delta$ many $\omega_1$ subsets added by the $(\mathbb{C} (\omega_1))^L$-part of the factors of $\forceP$. We let $\rho : ([\omega_1]^{\omega})^L \rightarrow \omega_1$ be our fixed, constructible bijection and let $h_{\beta}= \{ \rho (g_{\beta} \cap \alpha) \, : \, \alpha < \omega_1\}$. Note that the family $\{h_{\beta} \,: \, \beta < \delta \}$ forms an almost disjoint family of subsets of $\omega_1$. Thus there is $\alpha < \omega_1$ such that $\alpha> h_{\beta_1}\cap h_{\beta_2}$ for $\beta_1 \ne \beta_2 < \delta$ and additionally, $\alpha$ is an index not used by the iterated coding forcing $\forceP$, where we say that an index $i$ of $\vec{S}$ is used by $\forceP$ whenever an $\omega_1$-branch through $S_i$ is coded by a factor of $\forceP$.

We fix such an $\alpha$ and $S_{\alpha} \in \vec{S}$. We claim that there is no real in $W[G]$ such that $W[G] \models L[r] \models ``S_{\alpha}$ has an $\omega_1$-branch$"$.
We show this by pulling the forcing $S_{\alpha}$ out of $\forceP$. 
Indeed if we consider $W[\forceP]=L[\forceQ^0] [\forceQ^1][\forceQ^2][\forceP]$, and if $S_{\alpha}$ is as described already,
we can rearrange this to $W[\forceP]= L [\forceQ^0] [\forceQ'^1 \times S_{\alpha} ] [ \forceQ^2] [\forceP] = W[\forceP'] [S_{\alpha} ]$, where $\forceQ'^1$ is $\prod_{\beta \ne \alpha}  S_{\beta}$ and $\forceP'$ is $\forceQ^0 \ast \forceQ'^1 \ast \forceQ^2 \ast \forceP$.

Note now that, as $S_{\alpha}$ is $\omega$-distributive, $2^{\omega} \cap W[\forceP] = 2^{\omega} \cap W[\forceP']$, as $S_{\alpha}$ is still a Suslin tree in $W[\forceP']$ by the fact that $\vec{S}$ is independent, and no factor of $\forceP'$ besides the trees from $\vec{S}$ used in $\forceP'$ destroys Suslin trees. But this implies that 
\[W[\forceP'] \models \lnot \exists r L[r] \models `` S_{\alpha} \text{ has an $\omega_1$-branch}" \]
as the existence of an $\omega_1$-branch through $S_{\alpha}$ in the inner model $L[r]$ would imply the existence of such a branch in $W[\forceP']$. Further
and as no  new reals appear when passing to $W[\forceP]$ we also get 
\[W[\forceP] \models \lnot \exists r L[r] \models `` S_{\alpha} \text{ has an $\omega_1$-branch}". \]

On the other hand any unwanted information, i.e. any $(x,m,k) \notin \{(x_{\beta}, m_{\beta},k_{\beta}) \, : \, \beta < \delta \}$ such that $W[G] \models (x,m,k) \in A$, will also witness that \[L[r] \models `` S_{\alpha} \text{ has an $\omega_1$-branch}"\] for unboundedly many $\alpha$'s which are not in any of the $h_{\beta}$'s from above.

Indeed, if $r$ witnesses $({\ast} {\ast} {\ast})$ for $(x,y,m)$, then there must also be an uncountable $M$, $r \in M$, $M \models \ZFP$, whose local version of $L[r]$ believes $(\ast)$ for every $\gamma \in h$, as otherwise we could find $r, x \in M_0 \prec M$, $M_0$ countable, and the transitive collapse $\bar{M}_0$ of $M_0$ is a counterexample to the truth of $({\ast} {\ast} {\ast})$, which is a contradiction.

If $M$ is an uncountable, transitive $\ZFP$ model as above, then
$L[r]^M \models ``S_{\alpha}$  has an $\omega_1$-branch$"$, and as the trees from $\vec{S}$ are $\Sigma_1(\omega_1)$-definable, and as the existence of an $\omega_1$-branch is again a $\Sigma_1(\omega_1)$-statement, we obtain by upwards absoluteness that $L[r] \models ``S_{\alpha}$  has an $\omega_1$-branch$"$, as claimed.

In particular, as $(x,m,k ) \in A$, $r$ will satisfy that
\[n \in (x,y,m) \rightarrow L[r] \models ``S_{\omega \gamma+2n+1} \text{ has an $\omega_1$-branch}" \]
and
\[ n \notin (x,y,m) \rightarrow L[r] \models ``S_{\omega \gamma+2n} \text{ has an $\omega_1$-branch}". \] for $\omega_1$-many $\gamma$'s.

But by the argument above, only trees which we used in one of the factors of $\forceP$ have this property, so there can not be unwanted codes.


\end{proof}

\subsection{An auxiliary result}
We proceed via proving first the following auxiliary theorem whose proof introduces some of the key ideas, and will serve as a simplified blueprint for the proof of the main results.
\begin{theorem}
There is a generic extension $L[G]$ of $L$ in which there is a real $R_0$ such that every pair of disjoint (ligthface) $\Sigma^1_3$-sets can be separated by a $\Delta^1_3(R_0)$-formula.
\end{theorem}

For its proof, we will use the two easily definable $\omega_1$-sequences of Suslin trees on $\omega_1$, $\vec{S}=\vec{S}^1 \cup \vec{S}^2$ and branch shooting forcings to create for every pair $(A_m,A_k)$ of disjoint $\Sigma^1_3$-definable sets of reals a $\Delta_3^1 (\alpha_0)$-definable separating set $D_{m,k} \supset A_m$.
Using a bookkeeping function we list all the triples $(x,m,k)$ where $x$ is a real and $m, k\in \omega$, and decide for every such triple whether we code it into $\vec{S^1}$ which is equivalent to put it into $D_{m,k}$ or code it into $\vec{S^2}$ which eventually should become the complement $D_{m,k}^c$. Using coding arguments the sets $D_{m,k}$ and its complement will be $\Sigma^1_3(R_0)$-definable. The fact that we have to decide at every stage where to put the current real $x$ before the iteration is actually finished seems to be somewhat daring as the evaluation of the $\Pi^1_3$ and $\Sigma^1_3$-sets vary as we generically enlarge our surrounding universe along the iteration.
Additionally one has to deal with possible degenerated cases which stem from a certain amount of self referentiality in the way we set up things. Indeed it could happen that forcing a triple $(x,m,k)$ into one side, $D_{m,k}$ say, could force simultaneously that $x$ will become a member of $A_k$ in the generic extension, thus preventing $D_{m,k}$ to actually separate $A_m$ and $A_k$.
 A careful case distinction will show that this problem can be overcome though.

\subsection{Definition of the iteration over $W$}

For $n \in \omega$ let \[\varphi_n(v_0)= \exists v_1 \psi_n(v_0,v_1)\]  be the $n$-th formula in an enumeration of the $\Sigma^1_3$-formulas with one free variable. Let \[A_n:= \{ x \in 2^{\omega} \, : \, \varphi_n(x) \text{ is true} \},\] so $A_n$ is the set of reals whose definition uses the $n$-th $\Sigma^1_3$-formula in our enumeration. We force with an $\omega_1$-length
mixed support iteration of legal forcings which all have size $\aleph_1$, and use 
a definable, surjective bookkeeping-function \[ F: \omega_1 \rightarrow \omega_1  \times \omega_1 \times \omega \times \omega\]  to determine the iteration. We demand that every $\alpha < \omega_1$ is always strictly bigger than the first projection of $F(\alpha)$. We also assume that every quadruple $(\beta, \gamma, m,k)$ in $\omega_1  \times \omega_1 \times \omega \times \omega$ is hit unboundedly often by $F$.

The purpose of $F$ is to list all triples of the form $(x,m,k)$, where $x$ is a real in some intermediate universe of our iteration and $m,k \in \omega$ corresponds to a pair $(\varphi_m,\varphi_k)$ of $\Sigma^1_3$-formulas. The iteration will be defined in such a way, that, at every stage $\beta$ of the iteration, whenever some triple $(x,m,k)$ is considered by $F$, we must decide immediately whether to code (a real coding) the triple $(x,m,k)$ somewhere into the $\vec{S^1}$ or $\vec{S^2}$-sequence. The set of codes written into $\vec{S}^1$ which contain $m,k$ will result in the $\Sigma^1_3$-set $D^1_{m,k}$ which is a supset of $A_m$, the set of codes containing $m,k$ which are written into $\vec{S^2}$ shall result in the $\Sigma^1_3$-supset $D^2_{m,k}$ of $A_k$. The real $R_0$ will be used to indicate, in a uniform way for all $m,k$, the set of $\aleph_1$-many $\omega$-blocks of $\vec{S}$, which represent insecure data  we should not use for our separating sets.
The reader should think of $R_0$ as an error term, modulo which the separating sets will work.
That is $R_0 \in 2^{\omega}$ is such that it codes $\aleph_1$-many $\omega$-blocks of $\vec{S}$, and $D^1_{m,k}$ and $D^2_{m,k}$ should be the set of $(x,m,k)$'s which are coded into $\vec{S}^1$ and $\vec{S}^2$ respectively whose coding areas are almost disjoint from the set of ordinals coded by $R_0$, in that their intersection is bounded below $\omega_1$.
 
 More precisely 
let \begin{align*}
 D^1_{m,k} (R_0):= \{ (x,m,k) \, : &\, x \in 2^{\omega} \land (x,m,k) \text{ is coded into } \vec{S}^1  \\&\text{and its coding area is almost disjoint} \\&\text{ from the indices coded by $R_0$} \}
\end{align*}
and let $D^2_{m,k}(R_0)$ be defined similarly.  Our goal is to have $D^1_{m,k}(R_0) \cap D^2_{m,k}(R_0) = \emptyset$ for every $m,k \in \omega$. Thus we have found our separating sets for $A_m$ and $A_k$.

We proceed with the details of the inductive construction of the forcing iteration.
Assume that we are at some stage $\alpha < \omega_1$ of our iteration, let $\forceP_{\alpha}$ denote the partial order we have defined so far, let $G_{\alpha}$ denote a generic filter for $\forceP_{\alpha}$. We inductively assume in addition, that we have created a $\forceP_{\alpha}$-name of a set $\dot{b}_{\alpha}$ which is forced to be a set of countably many $\omega$-blocks of $\vec{S^1}$ and $\vec{S^2}$. Our goal is to define the next forcing $\dot{\forceQ}_{\alpha}$ which we shall use.
As will become clear after finishing the definition of the iteration, we can assume that $\forceP_{\alpha}$ is a legal notion of forcing.
We look at the value $F(\alpha)$ and define the forcing $\dot{\forceQ}_{\alpha}$
according to $F(\alpha)$ by cases as follows.
\subsubsection{Case a}

For the first case we assume that $F(\alpha) = ( \beta, \gamma, m,k)$, and that 
the $\gamma$-th (in some wellorder of $W$) name of a real of $W^{\forceP_{\beta}}$ is $\dot{x}$. 
We ask, whether there exists a forcing $\forceP$ such that
\[W[G_{\alpha}] \models \forceP \text{ is legal and } \forceP \Vdash \exists z (\varphi_m(z) \land\varphi_k(z)) .\] 
If there is such a legal $\forceP$, then we use it, i.e. we fix the $<$-least such forcing and let $\forceP(\alpha):= \forceP$, let $\forceP_{\alpha+1}=\forceP_{\alpha} \ast \forceP(\alpha)$, and let $G(\alpha+1)$ be $\forceP(\alpha+1)$-generic over $W$.

\subsubsection{Case b}
We assume again that $F(\alpha) = ( \beta, \gamma, m,k)$, and that 
the $\gamma$-th (in some wellorder of $W$) name of a real of $W^{\forceP_{\beta}}$ is $\dot{x}$. Let $\dot{x}^{G_{\alpha}}=x$.
Now we assume that case a is wrong, i.e. in $W[G_{\alpha}]$, there is no legal $\forceP$ such that $\forceP \Vdash \exists z (\varphi_m(z) \land \varphi_k(z) )$.
We shall distinguish three sub-cases.

\begin{itemize}
\item[(i)] First assume that there is a legal forcing $\forceQ$ such that
\begin{align*}
W[G_{\alpha}] \models &\, \forceQ \Vdash \varphi_m(x) 
\end{align*}
In this situation, we will code $(x,m,k)$ into the $\vec{S}^1$-sequence, i.e. we
let \[ \forceP(\alpha):= \hbox{Code} ((x,m,k),1) \]
and set $\forceP_{\alpha+1}:=\forceP_{\alpha} \ast \forceP(\alpha)$.
The upshot of the arguments above is the following:

 \begin{flushleft}
\textbf{Claim:} Let $G_{\alpha+1}$ be a $\forceP_{\alpha+1}$-generic filter over $W$ and let $\forceP$ be an arbitrary legal forcing in $W[G_{\alpha+1}]$. Then
\end{flushleft}
\[ \cancel{\Vdash}_{\forceP} \, x \in A_k.\]
\begin{proof}
Indeed if not, then pick $\forceP \in W[G_{\alpha+1}]$ such that there is a $p \in \forceP$ such that $p \Vdash_{\forceP} x \in A_k$. If we consider $\forceP$ below the condition $p$, we obtain a legal forcing $\forceP_{ \le p}$ again, and 
$\forceQ \times \forceP_{\le p} \Vdash x \in A_m  \cap A_k$, because $\forceQ$ introduces a real $r_m$ which witnesses that $\varphi_m(x)$ holds true. In particular $\varphi_m(x)$ is true in all outer models of $W[G_{\alpha+1}] [\forceQ]$ by upwards absoluteness of $\Sigma^1_3$-statements, which follows from Shoenfield absoluteness.

Likewise, $\forceP_{\le p}$ shows that $W[G_{\alpha+1}] [\forceP_{ \le p} \models \varphi_k (x)$, thus $\forceQ \times \forceP_{\le p}$ is a legal forcing which forces $x \in A_m \cap A_k$ which is a contradiction.
\end{proof}
\item[(ii)] The second subcase is symmetric to the first one. We assume that there is no legal forcing $\forceQ$, for which $\Vdash_{\forceQ} \varphi_m(x)$ is true, but there is a legal $\forceQ$ for which it is true that
\begin{align*}
W[G_{\alpha}] \models 
\forceQ \Vdash  \varphi_k(x). 
\end{align*}
Then we code $x$ into the $\vec{S^2}$-sequence with the usual coding. Note that by the symmetric argument from above, no further legal extension will ever satisfy that $x \in A_m$.
\item[(iii)] In the final subcase, there is no legal forcing which forces $x \in A_m \cup A_k$, and we are free to choose where to code $x$.
In that situation we settle to  code $(x,m,k)$ into $\vec{S^1}$.
\end{itemize}
This ends the inductive definition of the iteration.

\subsubsection{Discussion of case b}
We pause here to discuss briefly the crucial case b of the iteration. At first glance it seems promising, when in the first subcase of case b of the iteration, to use the legal forcing $\forceQ$, granted to exist by assumption, in order to obtain $x \in A_m$. After all, we know that case a is not true here, so if we can force $x \in A_m$ with a legal forcing, we can conclude that for all further future legal extensions, $x$ will not belong to $A_k$ which seems to fully settle the problem of where to place the particular triple $(x,m,k)$.

The just described strategy will fail however, for reasons having to do with the already mentioned self-referentiality of the the set-up. Indeed, one can easily produce $\Sigma^1_3$-predicates $\varphi_m$ and $\varphi_k$ such that case b will apply for all reals $x$, and such that whenever we decide to code $(x,m,k)$ into $\vec{S^1}$, in the resulting generic extension $\varphi_k(x)$ will become true. And vice versa, whenever we decide to code $(x,m,k)$ into $\vec{S^2}$, $\varphi_m(x)$ will hold in the resulting generic extension.
Thus, for these particular $m,k$, we are in case b throughout the iteration, and find a legal $\forceQ$ for every real $x$ we encounter, which forces $\forceQ \Vdash x \in A_k(x)$. But the forcings $\forceQ$, when applied, always add  pathological and unwanted situations, namely  $\forceQ \Vdash x \in A_k(x)$, yet $x$ is coded into $\vec{S^1}$. And as we have to place all the reals, we will produce these problems cofinally often throughout the whole iteration which ruins our attempts to proof the theorem.

Our definition of the iteration circumvents these problems via noting that the possibility to actually use the forcing $\forceQ$ is sufficient to rule out a pathological situation, by the closure of legal forcings under products. Thus the mere existence of such a legal $\forceQ$ is sufficient to not run into any problems when coding $(x,m,k)$ into $\vec{S}^1$.
We emphasize that this line of reasoning takes advantage of the specific coding method we decided to use and justifies the  construction of our ground model $W$.

\subsection{Discussion of the resulting W[G]}
We let $G$ be a generic filter for the $\omega_1$-length iteration which we just described using mixed support.
First we note that the iteration is proper, hence the iteration preserves $\aleph_1$. Consequently there will be no new reals added at stage $\omega_1$, so $\omega^{\omega} \cap W[G] = \bigcup_{\alpha< \omega_1} \omega^{\omega} \cap W[G_{\alpha}]$,  in particular $\CH$ is true in $W[G]$.

A second useful observation is that for every pair of stages $\alpha <\beta < \omega_1$, the quotient-forcing which we use to pass from $\forceP_{\alpha}$ to $\forceP_{\beta}$ is a legal forcing as seen from the intermediate model $W[\forceP\alpha]$.  

Our goal is now to define a real $R_0$ and, given a pair of disjoint $\Sigma^1_3$-definable sets $A_m,A_k$, a $\Delta^1_3(R_0)$-definable separating set, i.e. a set such that $A_m \subset D^1_{m,k}$ and $A_k \subset D_{m,k}^2$ and such that $D_{m,k}^2=D_{m,k}^c$.
We want our set $D^1_{m,k} (R_0)$ to consist of the codes written into $\vec{S}^1$ beyond $\alpha_0$ which itself contain a code for the pair $(m,k)$ and its converse $D_{m,k}^2$ to consist of all the codes on the $\vec{S}^2$-side  which contain a code for $(m,k)$ and whose coding areas are almost disjoint from the subset of $\omega_1$ coded by the real $R_0$. What should the real $R_0$ be?
It is clear from the definition of the iteration $\forceP$, that there are stages in the iteration where case a applies.  There, we just blindly use legal forcings. In particular, nothing prevents these legal forcings to code up $(x,m,k)$ into, say $\vec{S^1}$ while $\varphi_k(x)$ is true, thus adding a problem.

Note however that such degenerate situations can only happen once for every pair $(A_m,A_k)$. As we only have countably many such pairs and as our iteration has length $\omega_1$ and as we visit every triple $(x,m,k)$ uncountably often with our bookkeeping function, there will be a stage $\beta_0 < \omega_1$ such that from $\beta_0$ on all the codes we have written into $\vec{S}^1$ and $\vec{S}^2$ are intended ones, i.e. the codes really define a separating set $D_{m,k}$ for $A_m$ and $A_k$.

Thus, in order to define $R_0$, we first let $\beta_0 < \omega_1$ be the last stage in the iteration $\forceP$, where case a is applied. Then, working in $W[G_{\beta_0}]$, we let $R_0$ code the collection of all indices of all the trees from $\vec{S^1}$ and $\vec{S^2}$, which were used for coding in $W[G_{\beta_0}]$. Note here that this collection is characterized by the countable set $\{ r_{\beta} \, : \, \beta < \beta_0\}$ where $r_{\beta}$ is the real which is added with almost disjoint coding at stage $\beta$ and which witnesses $( {\ast} {\ast} {\ast} )$ for $(x_{\beta},m_{\beta},k_{\beta})$ and each $\gamma \in h_{\beta}$ (where $h_{\beta}$ just denotes the coding area given by the real $r_{\beta}$) holds. This countable set of reals can itself be coded by a real, and this real is $R_0$.

We define:
\begin{align*}
 x \in D_{m,k}^1(R_0) \Leftrightarrow &\exists r \in 2^{\omega} L[r] \models  ``\exists M  (\text{$M$ witnesses that }(\ast) \text{ holds for $(x,m,k)$,  $\vec{S}^1$}\\ & \text{and every $\gamma$ in its coding area $h \subset \omega_1"$}. 
  \\& \text{ Further } L[r,R_0] \models ``h  \text{ is almost disjoint} \\&  \qquad \qquad \text{ from the set of indices coded by $R_0"$ )}
\end{align*}
and
\begin{align*}
  x \in D_{m,k}^2(R_0) \Leftrightarrow &\exists r \in 2^{\omega} L[r] \models  ``\exists M  (\text{$M$ witnesses that }(\ast) \text{ holds for $(x,m,k)$,  $\vec{S}^2$}\\ & \text{and every $\gamma$ in its coding area $h \subset \omega_1"$}. 
  \\& \text{ Further } L[r,R_0] \models ``h  \text{ is almost disjoint} \\&  \qquad \qquad \text{ from the set of indices coded by $R_0"$ )}
\end{align*}
We shall show now that these sets work as intended.
\begin{lemma}
In $W[G]$ for every pair $m \ne k \in {\omega}$, $D^1_{m,k}(R_0)$ and $D_{m,k}^2(R_0)$ union up to all the reals.
\end{lemma}
\begin{proof}
Immediate from the definitions.
\end{proof}

\begin{lemma}
In $W[G]$ for every pair $(m,k)$, if the $\Sigma^1_3$-sets $A_m$ and $A_k$ are disjoint then $D^1_{m,k}(R_0)$  separates $A_m$ from $A_k$, i.e. $A_m \subset D^1_{m,k}(R_0)$ and $A_k \cap D^1_{m,k}(R_0)=\emptyset$. Likewise $D^2_{m,k}$ separates $A_k$ from $A_m$. Consequentially, for every $m,k$ such that $A_m \cap A_k = \emptyset$, $D^1_{m,k}(R_0) \cap D^2_{m,k}(R_0)=\emptyset$.
\end{lemma}
 \begin{proof}
We will only proof the first assertion, the second one is proved exactly as the first one with the roles of $m,k$ switched. Assume that $A_m$ and $A_k$ are disjoint and let $x \in W[G] \cap 2^{\omega}$ be arbitrary, such that $x \in A_m$ is coded into $\vec{S^1}$ and its coding area is almost disjoint from the set of ordinals coded by $R_0$.
 There is a least stage $\alpha$ with $\beta_0<\alpha< \omega_1$ such that $F(\alpha)= (\dot{x},m,k)$ where $\dot{x}$ is a name for $x$. According to the definition of the iteration and the assumption that $A_m \cap A_k=\emptyset$, we can rule out case a. Thus case b remains, and hence the first or the third subcase did apply at stage $\alpha$.  Suppose, without loss of generality, that we were in the first subcase of case b. 
Assume for a contradiction that in $W[G]$, $x \in A_k$, then there would be a stage $\alpha'$ of the iteration, $\alpha'> \alpha> \beta_0$ such that  $W[G_{\alpha'}] \models x \in A_k$ and the part of the iteration $\forceP$ between stage $\alpha$ and $\alpha'$, denoted with $\forceP_{\alpha \alpha'}$, is a legal forcing, and which forces $x \in A_k$. But, as at stage $\alpha$, the first subcase of b applied, there is a legal forcing $\forceQ$, such that $\forceQ \Vdash x \in A_m$, hence, at $\alpha$, there is a legal forcing 
 which forces $x \in A_m \cap A_k$, namely $\forceP_{\alpha \alpha'} \times \forceQ$, which is a contradiction.

 \end{proof}

\begin{lemma}\label{Sigma13}
In $W[G]$, for every $m,k \in \omega$, $D^1_{m,k}$ and $D^2_{m,k}$ are $\Sigma^1_3(R_0)$-definable. Thus $W[G]$ satisfies that every pair of disjoint $\Sigma^1_3$-sets can be separated by a $\Delta^1_3 (R_0)$-set.
\end{lemma}

\begin{proof}
 We claim that for $m,k \in \omega \times \omega$ arbitrary, $D^1_{m,k}$ and $D^2_{m,k}$ have the following definitions in $W[G]$:
 
 \begin{align*}
 x \in D_{m,k}^1 \Leftrightarrow & \exists r \forall M (r, R_0 \in M \land \omega_1^M=(\omega_1^L)^M \land M \text{ transitive } \rightarrow  \\ &M \models L[r] \models ``\exists N ( N \models \ZFP \land |N| = \aleph_1^M  \land N \text{ is transitive } \land \\ &N \text{ believes $(\ast)$ for $(x,y,m)$ and $\vec{S}^1$ and every $\gamma \in h"$} \\&\text{ and $L[r,R_0] \models``$the coding area $h$ of $(x,m,k)$ is almost disjoint} \\&
 \text{from the set of indices coded by $R_0"$} )).
\end{align*}
and
\begin{align*}
x \in D_{m,k}^2 \Leftrightarrow & \exists r \forall M (r, R_0 \in M \land \omega_1^M=(\omega_1^L)^M \land M \text{ transitive } \rightarrow  \\ &M \models L[r] \models ``\exists N ( N \models \ZFP \land |N| = \aleph_1^M  \land N \text{ is transitive } \land \\ &N \text{ believes $(\ast)$ for $(x,y,m)$ and $\vec{S}^2$ and every $\gamma \in h"$} \\&\text{ and $L[r,R_0] \models``$the coding area $h$ of $(x,m,k)$ is almost disjoint} \\&
 \text{from the set of indices coded by $R_0"$} )).
\end{align*}
 
Counting quantifiers yields that both formulas are of the form $\exists \forall (\Sigma^1_2 \rightarrow \Delta^1_2)$ and hence $\Sigma^1_3$.

We will only show the result for $D_{m,k}^1$.
To show the direction from left to right, note that if $x \in D_{m,k}^1$, then 
there was a stage $\alpha>\beta_0$ in our iteration such that we coded $x$ into the $\vec{S}^1$-sequence. In particular we added a real $r_{\alpha}$ for which property $({\ast}{\ast}{\ast} )$ is true, hence $r_{\alpha}$ witnesses that the right hand side is true in $W[G]$.

For the other direction assume that the right hand side is true. This in particular means that the assertion is true for transitive models containing $r$ of arbitrary size. Indeed if there would be a transitive $M$ which contains $r$ and whose size is $\ge \aleph_1$, then there would be a countable $M_0 \prec M$ which contains $r$. The transitive collapse of $M_0$ would form counterexample to the assertion of the right hand side, which is a contradiction to our assumption.

But if the right hand side is true for models of arbitrary size, by reflection it must be true for $W[G]$ itself, hus $x \in D_{m,k}^1$, and we are done. 

\end{proof}

\section{Boldface Separation}
\subsection{Preliminary Considerations}
We turn our attention to boldface separation. Goal of this section is to prove the first main theorem.
\begin{theorem}
There is an $\omega_1$-preserving, generic extension of $L$ in which every pair of disjoint $\bf{\Sigma^1_3}$-sets $A_m$ and $A_k$ can be separated by a $\bf{\Delta^1_3}$-set.
\end{theorem}
It uses the proof of our auxiliary theorem as the base case of an inductive construction. The main idea to keep control is to replace the notion of legal forcing with a dynamic variant which keeps changing along the iteration.

To motivate the following we first consider a more fine-tuned approach to the definition of the iteration of the proof of the last theorem. Let us assume that $(m,k)$ is the first pair such that case b in the definition of the iteration applies. Recall that in the discussion of case b, we showed that, given a pair $(A_m, A_k)$ of $\Sigma^1_3$-sets for which there does not exist a legal forcing $\forceQ$ such that $\Vdash_{\forceQ} \exists z (\varphi_m(z) \land \varphi_k(z))$ becomes true, we can assign for an arbitrary real $x$ always a side $\vec{S}^1$ or $\vec{S}^2$ such that in all future legal extensions, there will never occur a pathological situation, i.e. from that stage on we never run into the problem of having coded the triple $(x,m,k)$ into, say, $\vec{S^1}$, yet $\varphi_k(x)$ becomes true in some future extension of our iteration (or vice versa). Note here that the arguments in the discussion of case b were uniform for all reals $x$ which appear in a legal extension. So it is reasonable to define for the pair $(m,k)$ 
a stronger notion of legality, called 1-legal with respect to $(0,m,k)$ (the 0 indicates the base case of an inductive construction we define later) as follows:

Let $F: \gamma \rightarrow H(\omega_2)^4$ be a bookkeeping function and let $E:=\{ (0,m,k) \}$. We let $\forceP$ be a mixed support iteration of length $ \gamma$. Then we say that $(\forceP  \,  : \, \beta < \gamma\} ) $ is 1-legal with respect to $E$ and $F$ if
\begin{itemize}
\item $\forceP$ is a legal forcing relative to $F$.
\item Whenever $\beta < \gamma$ is a stage such that $F(\beta)=(\dot{x},m,k,i)$, where $\dot{x}$ is a $\forceP_{\beta}$-name of a real, $\xi$ is an ordinal and $i \in \{1,2 \}$ we split into three subcases:
\begin{itemize}
\item[(i)] First we assume that in $W[G_{\beta}]$,
there is a legal forcing $\forceQ$ such that $\forceQ \Vdash x (=\dot{x}^{G_{\beta}}) \in A_m$.

Then, the $\beta$-th forcing of $\forceP$, $\forceP(\beta)$ must be $\hbox{Code} ((x,m,k), 1)$.
\item[(ii)]
Assume that (i) is wrong but the dual situation is true for $A_k$. That is, 
there is a legal forcing $\forceQ$ such that $\forceQ \Vdash x (=\dot{x}^{G_{\beta}}) \in A_k$.

Then, the $\beta$-th forcing of $\forceP$, $\forceP(\beta)$ must be $\hbox{Code} ((x,m,k), 2)$.
\item[(iii)]

In the third case  we assume that neither (i) nor (ii) is true.
In that situation we force with either coding $x$ into the $A_m$ or the $A_k$ side, whatever the bookkeeping tells us.


\end{itemize}

If the iteration  $(\forceP \, : \, \beta < \gamma)$ obeys the above rules, then we say that it is 1-legal with respect to $E$ and bookkeeping function $F$.
\end{itemize}

From earlier considerations it is clear that if we drop the notion of legal from now on and replace it with 1-legal relative to $E=\{ (0,m,k)\}$ in our iteration,  we can ensure that, at least for the pair $(m,k)$ no new pathological situations will arise anymore.
This process can be iterated. Assume that we run into a new pair $(m',k')$ where the modified case b applies, i.e. there is no 1-legal forcing which forces $\exists z (\varphi_{m'}(z) \land\varphi_{k'} (z))$, we can introduce the new notion of 2-legal with respect to $\{(0,m,k) , (1,m',k')\}$. If chosen the right way, this new notion will hand us a condition that guarantees that no new pathological situations arises for the two pairs $(\varphi_m,\varphi_k)$ and $(\varphi_{k'}, \varphi_{m'})$. 

So our strategy for producing a model where the $\bf{\Sigma^1_3}$-separation property is as follows: we list all possible reals $x$, parameters $y$ and pairs of $\Sigma^1_3$-formulas
$(\varphi_m(\cdot,y), \varphi_k(\cdot,y))$, while simultaneously define stronger and stronger versions of legality, which take care of placing the reals we encounter along the iteration in a non-pathological way.

\subsection{$\alpha$-legal forcings}
This section shall give a precise recursive definition of the
process sketched above.

The notions of 0 and 1-legality will form the base cases of an inductive definition.
Let $\alpha \ge 1$ be an ordinal and assume we defined already the notion of $\alpha$-legality. Then we can inductively define the notion of $\alpha+1$-legality  as follows.

Suppose that $\gamma < \omega_1$, $F$ is a bookkeeping function,  \[F: \gamma \rightarrow H(\omega_2)^5 \] and \[\forceP=(\forceP_{\beta} \, : \, \beta < \gamma)\] is a legal forcing relative to $F$ (in fact relative to some bookkeeping $F'$ determined by $F$ in a unique way - the difference here is not relevant). 

Suppose that 
\[E= \{(\delta, \dot{y}_{\delta}, m_{\delta} ,k_{\delta}) \, : \,  \delta \le \alpha\}\]
where $m_{\delta},k_{\delta} \in \omega$ and every $\dot{y}_{\delta}$ is a  $\forceP$-name of a real and for every two ordinals $\beta, \gamma< \alpha$, $\forceP \Vdash (\dot{y}_{\beta}, m_{\beta},k_{\beta}) \ne (\dot{y}_{\gamma},m_{\gamma},k_{\gamma})$. Suppose that for every $\delta \le \alpha$, $(\forceP_{\beta} \, : \, \beta < \gamma )$ is $\delta$-legal with respect to $E \upharpoonright \delta = \{ (\eta, \dot{y}_{\eta},m_{\eta},k_{\eta}) \in E \, : \, \eta < \delta \}$ and $F$.
Finally assume that $\dot{y}_{\alpha+1}$ is a $\forceP$-name for a real and $m_{\alpha+1},k_{\alpha+1} \in \omega$ such that $\forceP \Vdash \forall \delta \le \alpha ((\dot{y}_{\delta},m_{\delta},k_{\delta}) \ne (\dot{y}_{\alpha+1},m_{\alpha+1},k_{\alpha+1}))$. Then we say that $(\forceP_{\beta} \, : \, \beta < \gamma )$ is $\alpha+1$-legal with respect to $E \cup \{\alpha+1, \dot{y}_{\alpha+1},m_{\alpha+1},k_{\alpha+1})\}$ and $F$ if it obeys the following rules. 
\begin{enumerate}
\item Whenever $\beta < \gamma $ is such that there is a $\forceP_{\beta}$-name $\dot{x}$ of a real and an integer $ i\in\{1,2\}$ such that
\[F(\beta)= (\dot{x},\dot{y}_{\alpha+1},m_{\alpha+1},k_{\alpha+1},i)\] and $\dot{y}_{\alpha+1}$ is in fact a $\forceP_{\beta}$-name, and for $G_{\beta}$  a $\forceP_{\beta}$-generic over $W$, $W[G_{\beta}]$ thinks that
\begin{align*}
\exists \forceQ (&\forceQ \text{ is } \alpha \text{-legal with respect to } E \,  \land \\& \forceQ \Vdash x \in A_m({y}_{\alpha+1})),
\end{align*}
where $x=\dot{x}^G$, and $y_{\alpha}=\dot{y}_{\alpha+1}^G$.
Then continuing to argue in $W[G_{\beta}]$, if ${\forceQ}_1=\dot{\forceQ}_1^{G_{\beta}}$ we let
\[\forceP(\beta)= \hbox{Code}((x,y,m,k),1). \]  Note that we confuse here the quadruple $(x,y,m,k)$ with one real $w$ which codes this quadruple.

\item
Whenever $\beta < \gamma $ is such that there is a $\forceP_{\beta}$-name $\dot{x}$ of a real and an integer $ i \in \{1, 2\}$ such that
\[F(\beta)= (\dot{x},\dot{y}_{\alpha+1},m_{\alpha+1},k_{\alpha+1}, i)\] and for $G_{\beta}$ which is $\forceP_{\beta}$-generic over $W$, $W[G_{\beta}]$ thinks that
\begin{align*}
\forall \forceQ_1 (&\forceQ_1 \text{ is } \alpha \text{-legal with respect to } E   \\& \rightarrow \, \lnot(\forceQ_1 \Vdash x \in A_m(\dot{y}_{\alpha+1})))
\end{align*}
but there is a forcing $\forceQ_2$ such that $W[G_{\beta}]$ thinks that 
\begin{align*}
\forceQ_{2} \text{ is } \alpha &\text{-legal with respect to $E$ and } \\ & {\forceQ_2} \Vdash x \in A_k( \dot{y}_{\alpha+1} )
\end{align*}

Then continuing to argue in $W[G_{\beta}]$, we force with
\[\forceP(\beta):= \hbox{Code}((x,y,m,k),2).\]  Note that we confuse here again the quadruple $(x,y,m,k)$ with one real $w$ which codes this quadruple.

\item If neither 1 nor 2 is true, then either \[ \forceP(\beta)=\hbox{Code}((x,y,m,k),2)\] or 
\[ \forceP(\beta)=\hbox{Code}((x,y,m,k),1)\] depending on whether $ i\in\{1,2\}$ in $F(\beta)$ was 1 or 2. 

\item 
If $F(\beta) = (\dot{x},\dot{y},m,k,i)$ and for our $\forceP_{\beta}$-generic filter $G$, $W[G] \models \forall \delta \le \alpha+1 ((\delta,\dot{y}^G,m,k) \notin E^G)$,
then, working over $W[G_{\beta}]$ let
\[ \forceP(\beta)=\hbox{Code}((x,y,m,k),i)\] depending on whether $ i\in\{1,2\}$ in $F(\beta)$ was 1 or 2.
\end{enumerate} 
This ends the definition for the successor step $\alpha \rightarrow \alpha+1$.
For limit ordinals $\alpha$, we say that a legal forcing $\forceP$ is $\alpha$ legal with respect to $E$ and $F$ if for every $\eta < \alpha$,
$(\forceP_{\beta} \,: \, \beta < \gamma)$ is $\eta$-legal with respect to $E \upharpoonright \eta$ and some $F'$.
\par \medskip

We add a couple of remarks concerning the last definition.
\begin{itemize}

\item By definition, if $\delta_2 < \delta_1$ and $\forceP_1$ is $\delta_1$-legal with respect to $E= \{(\beta, \dot{y}_{\beta}, m_{\beta} ,k_{\beta}) \, : \,  \beta \le \delta_1\}$ and some $F_1$, then  $\forceP_1$ is also $\delta_2$-legal with respect to $E \upharpoonright \delta_2 = \{(\beta, \dot{y}_{\beta}, m_{\beta} ,k_{\beta}) \, : \,  \beta \le \delta_2\}$ and an altered bookkeeping function $F'$.
 
\item The notion of $\alpha$-legal can be defined in a uniform way over any legal extension $W'$ of $W$.

\item We will often just say that some iteration $\forceP$ is $\alpha$-legal, by which we mean that there is a set $E$ and a bookkeeping $F$ such that $\forceP $ is $\alpha$-legal with respect to $E$ and $F$.
\end{itemize}

\begin{lemma}\label{productlegal}
Let $\alpha \ge 1$, assume that $W'$ is some $\alpha$-legal generic extension of $W$, and that $\forceP^1=(\forceP^1_{\beta} \,: \,\beta < \delta)$ and $\forceP^2=(\forceP^2_{\beta} \,: \, \beta < \delta) $ are two $\alpha$-legal forcings over $W'$ with respect to a common set $E=\{\delta,\dot{y}_{\delta},m_{\delta}, k_{\delta} \,: \, \delta < \alpha\}$ and bookkeeping functions $F_1$ and $F_2$ respectively. 
Then there is a bookkeeping function $F$ such that $\forceP_1 \times \forceP_2$ is $\alpha$-legal over $W'$ with respect to $E$ and $F$.
\end{lemma}
\begin{proof}
We define $F \upharpoonright \delta_1$ to be $F_1$. For values $\delta_1+ \beta > \delta_1$ we let $F(\delta_1+\beta)$ be such that its value on the first four coordinates equal the first four coordinates of $F_2(\beta)$, i.e. $F(\delta_1+\beta)=(\dot{x},\dot{y},m,k,i)$ for some $i \in \{1,2\}$ where $F_2(\beta)=(\dot{x},\dot{y},m,k,i')$. We claim now that we can define the remaining value of $F(\beta)$, in such a way that the lemma is true. This is shown by induction on $\beta< \delta_2$.
Let $(\forceP_2)_{\beta}$ be the iteration of $\forceP_2$ up to stage $\beta < \delta_2$. Assume, that $\forceP_1 \times (\forceP_2)_{\beta}$ is in fact an $\alpha$-legal forcing relative to $E$ and $F$. Then we have that $F(\delta_1+\beta) \upharpoonright 5=F_2(\beta) \upharpoonright 5 =(\dot{x},\dot{y},m,k)$, and we claim that at that stage,
\begin{claim} 
If we should apply case 1,2 or 3, when considering the forcing $\forceP_1 \times \forceP_2$ as an $\alpha$-legal forcing relative to $E$ over the model $W'$, we must apply the same case when considering $\forceP_2$ as an $\alpha$-legal forcing over the model $W'$ relative to $E$.
\end{claim}   
Once the claim is shown, the lemma can be proven as follows by induction on $\beta < \delta_2$: we work in the model $W'[\forceP_1][(\forceP_2)_{\beta}]$, consider $F(\delta_1+\beta) \upharpoonright 5 = F_2(\beta) \upharpoonright 5$, and ask which of the four cases has to be applied. By the claim, it will be the same case, as when considering $\forceP_2$ over $W'$ as an $\alpha$-legal forcing relative to $E$ and $F_2$. In particular the forcing $\forceP_2(\beta)$ we define at stage $\beta$ will be a choice obeying the rules of $\alpha$-legality, even when working over the model $W'[\forceP_1][(\forceP_2)_{\beta}]$.  This shows that $\forceP_1 \times \forceP_2$ is an $\alpha$-legal forcing relative to $E$ and some $F$ over $W'$.

\medskip

The proof of the claim is via induction on $\alpha$.
If $\alpha=1$ and both $\forceP_1$ and $\forceP_2$ are 1-legal with respect to $E$ which must be of the form $ \{0,\dot{y},m,k\}$, then we shall show that there is a bookkeeping $F$ such that $(\forceP_2)_{\beta}\, : \, \beta < \delta_2\})$ is still 1-legal with respect to $E$, even when considered in the universe $W'[\forceP_1]$.
We assume first that at stage $\delta_1+\beta$ of $\forceP_1 \times \forceP_2$  case 1 in the definition of 1-legal applies, when working in the model $W'[\forceP_1][(\forceP_2)_{\beta}]$ relative to $E$ and $F$. Thus
\[ F(\beta) \upharpoonright 5= (\dot{x},\dot{y},m,k) \]
and $(0,\dot{y},m,k) \in E$ and for any $G^1 \times G_{\beta}$ which is $\forceP_1 \times (\forceP_2)_{\beta}$-generic over $W'$, if $\dot{x}^{G_{\beta}}=x$ and $\dot{y}^{G_{\beta}}=y$, the universe $W'[G_1 \times G_{\beta}]$ thinks that
\begin{align*}
\exists \forceQ (&\forceQ \text{ is } 0 \text{-legal with respect to } E \text{ and some F }\,  \land \\& \, \forceQ \Vdash {x} \in A_m({y})).
\end{align*}
Thus, if we work over $W'[G_{\beta}]$ instead it will think
\begin{align*}
\exists (\forceP_1 \times \forceQ) (&\forceP_1 \times \forceQ \text{ is } 0 \text{-legal } \,  \land \\&  \, \forceP_1 \times \forceQ \Vdash {x} \in A_m({y})).
\end{align*}
Thus, at stage $\beta$, we are in case 1 as well, when considering $\forceP_2$ as an 1-legal forcing over $W'$ relative to $E$.

If, at stage $\beta$, case 2 applies, when considering $\forceP_1 \times \forceP_2$ as a 1-legal forcing with respect to $E$ over $W'$, then we argue first that case 1 is impossible when considering $\forceP_2$ as a $1$-legal forcing over $W'$.
Indeed, assume for a contradiction that case 1 must be applied, then, by assumption,
$\forceP_2(\beta)$ will force that $x \in A_m(y)$. Yet, by Shoenfield absoluteness, $\forceP_2(\beta)$ would witness that we are in case 1 at stage $\beta$ when considering $\forceP_1 \times\forceP_2$ as 1-legal with respect to $E$ over $W'$, which is a contradiction.

Thus we can not be in case 1 and we shall show that we are indeed in case 2, i.e. there is a 0-legal forcing $\forceQ$, such that $\forceQ \Vdash x \in A_k(y)$, but such a $\forceQ$ exists, namely $\forceP_2(\beta)$, 

Finally, if at stage $\beta$, case 3 applies when considering $\forceP_2$ as a 1-legal forcing with respect to $E$ over $W'[\forceP_1]$, we claim that we must be in case 3 as well, when considering $\forceP_2$ over just $W'$.
If not, then we would be in case 1 or 2 at $\beta$. Assume without loss of generality that we were in case 1, then, as by assumption $\forceP_2$ is 1-legal over $W'$, $\forceP_2(\beta)$ will force $\Vdash x \in A_m(y)$. But this is a contradiction, so we must be in case 3 as well.
This finishes the proof of the claim for $\alpha=1$.

\medskip

We shall argue now that the Claim is true for $\alpha+1$-legal forcings provided we know that it is true for $\alpha$-legal forcings.
Again we shall show the claim via induction on $\beta$. So assume that $\forceP_1 \times (\forceP_2)_{\beta}$ is $\alpha+1$-legal with respect to $E=E \upharpoonright \alpha \cup \{(\alpha, \dot{y},m_{\alpha},k_{\alpha})\}$ and an $F$ whose domain is $\delta_1 +\beta$. 
We look at
\[F(\delta_1+ \beta) \upharpoonright 5=F_2(\beta) \upharpoonright 5=  (\dot{x},\dot{y},m_{\alpha},k_{\alpha})\]
We concentrate on the case where $\beta$ is such that case 2 applies when considering $\forceP_1 \times (\forceP_2)_{\beta}$ over $W'$. The rest follows similarly. Our goal is to show that case 2 must apply when considering the $\beta$-th stage of the forcing using $F_2$ and $E$ over $W'[(\forceP_2)_{\beta}]$ as well.

Assume first for a contradiction, that, when working over $W'[(\forceP_2)_{\beta}]$, at stage $\beta$, case 1 applies.
Then, for any $(\forceP_2)_{\beta}$-generic filter $ G_{\beta}$ over $W'$, 
\begin{align*}
W'[G_{\beta}] \models \exists \forceQ (&\forceQ \text{ is $\alpha$-legal with respect to $E \upharpoonright \alpha$ and some $F'$ and} \\ & \forceQ \Vdash x \in A_m(y))
\end{align*}
Now, as $\forceP_2$ is $\alpha$-legal, we know that $\forceP_2(\beta)$ is such that $\forceP_2(\beta) \Vdash x \in A_m(y)$.

Thus, using the upwards-absoluteness of $\Sigma^1_3$-formulas, at stage $\beta$ of the $\alpha+1$-legal forcing determined by $F$ and $E$, there is an $\alpha$-legal forcing $\forceQ$ with respect to $E \upharpoonright \alpha$ which forces $x \in A_m(y)$, namely $\forceP_2(\beta)$. But this is  a contradiction, as we assumed that when considering $\forceP_1 \times (\forceP_2)_{\beta}$ over $W'$ at stage $\beta$, case 1 does not apply, hence such an $\alpha$-legal forcing should not exist.

So we know that case 1 is not true. We shall show now that case 2 must apply at stage $\beta$ when considering $\forceP_2$ over the universe $W'$. By assumption we know that
\begin{align*}
 W'[\forceP_1][(\forceP_2)_{\beta}] \models&\exists \forceQ_2 (\forceQ_2 \text{ is $\alpha$-legal with respect to $E \upharpoonright \alpha$ and } \\&  \, \, \forceQ_2 \Vdash x \in A_k(y)
\end{align*}
As $\forceP_1$ is $\alpha+1$-legal with respect to $E$ and $F_1$, it is also $\alpha$-legal with respect to $E \upharpoonright \alpha$ and some altered $F'_1$, thus, as a consequence from the induction hypothesis, we obtain that 
\[W'[(\forceP_2)_{\beta}] \models \forceP_1 \times \forceQ_2 \text{ is $\alpha$-legal and } \forceP_1 \times \forceQ_2 \Vdash x \in A_k(y).\] 
But then, $\forceP_1 \times \forceQ_2$ witnesses that we are in case 2 as well when at stage $\beta$ of $\forceP_2$ over $W'$.
This ends the proof of the claim and so we have shown the lemma.
\end{proof}
\subsection{Proof of the first Main Theorem}
We are finally in the position to prove that the $\bf{\Sigma^1_3}$-separation property can be forced over $W$.
The iteration we are about to define inductively will be a legal iteration, whose tails are $\alpha$-legal and $\alpha$-increases along the iteration. We start with fixing a bookkeeping function \[ F: \omega_1 \rightarrow H(\omega_1)^4 \]
which visits every element cofinally often. The role of $F$ is to list all the quadruples of the form $(\dot{x}, \dot{y},m,k)$, where $\dot{x}, \dot{y}$ are names of reals in the forcing we already defined, and $m$ and $k$ are natural numbers which represent  $\Sigma^1_3$-formulas with two free variables, cofinally often.
Assume that we are at stage $\beta < \omega_1$ of our iteration.  By induction we will have constructed already the following list of objects.
\begin{itemize}

\item An ordinal $\alpha_{\beta} \le \beta$ and a set $E_{\alpha_{\beta}}$ which is of the form $\{\eta ,\dot{y}_{\eta}, m_{\eta},k_{\eta} \, : \, \eta < \alpha_{\beta} \}$, where $\dot{y}_{\eta}$ is a $\forceP_{\beta}$-name of a real, $m_\eta, k_{\eta}$ are natural numbers. As a consequence, for every bookkeeping function $F'$, we do have a notion of $\eta$-legality relative to $E$ and $F'$ over $W[G_{\beta}]$.

\item We assume by induction that for every $\eta < \alpha_{\beta}$, if $\beta_{\eta}< \beta$ is the $\eta$-th stage in $\forceP_{\beta}$, where we add a new member to $E_{\alpha_{\beta}}$, then $W[G_{\beta_{\eta}}]$ thinks that the $\forceP_{\beta_{\eta} \beta}$ is $\eta$-legal with respect to $E_{\alpha_{\beta}} \upharpoonright \eta$.

\item If $(\eta, \dot{y}_{\eta},m_{\eta},k_{\eta}) \in E_{\alpha_{\beta}}$, then we set again $\beta_{\eta}$ to be the $\eta$-th stage in $\forceP_{\beta}$ such that a new member to $E_{\alpha_{\beta}}$ is added. In the model $W[G_{\beta_{\eta}}]$, we can form the set of reals $R_{\eta}$ which were added so far by the use of a coding forcing in the iteration up to stage $\beta_{\eta}$, and  which witness $({\ast} {\ast} {\ast})$ holds for some $(x,y,m,k)$; 

Note that $R_{\eta}$ is a countable set of reals and can therefore be identified with a real itself, which we will do. The real $R_{\eta}$ indicates the set of places we must avoid when expecting correct codes, at least for the codes which contain $\dot{y}_{\eta},m_{\eta}$ and $k_{\eta}$.

\end{itemize}
Assume that $F(\beta)= (\dot{x},\dot{y},m,k)$, assume that $\dot{x}$, $\dot{y}$ are $\forceP_{\beta}$-names for reals, and  $m,k\in \omega$ correspond to the $\Sigma^1_3$-formulas $\varphi_m(v_0,v_1)$ and $\varphi_k(v_0,v_1)$. Assume that $G_{\beta}$ is a $\forceP_{\beta}$-generic filter over $W$. Let $\dot{x}^{G_{\beta}}=x$ and $\dot{y}_1^{G_{\beta} }=y_1, \dot{y}_2^{G_{\alpha}}=y_2$. 
We turn to the forcing $\forceP(\beta)$ we want to define at stage $\beta$ in our iteration.
Again we distinguish several cases.
\begin{itemize}
\item[(A)]  Assume that $W[G_{\beta}]$ thinks that there is an $\alpha_{\beta}$-legal forcing $\forceQ$ relative to $E_{\alpha_{\beta}}$ and some $F'$ such that 
\begin{align*}
\forceQ \Vdash 
\exists z (z\in A_m(y) \cap A_k(y)).
\end{align*}
Then we pick the $<$-least such forcing, where $<$ is some previously fixed wellorder. We denote this forcing with $\forceQ_1$
and use 
\[\forceP(\beta):= \forceQ_1.\]
We do not change $R_{\beta}$ at such a stage.

\item[(B)] Assume that (i) is not true.

\begin{itemize} 

\item[(i)] Assume however that there is an $\alpha_{\beta}$-legal forcing $\forceQ$ in $W[G_{\beta}]$ with respect to $E_{\alpha_{\beta}}$ and some $F'$ such that 
\begin{align*}
\forceQ \Vdash 
 x \in A_m(y).
\end{align*}
Then we set 
\[\forceP(\beta):= \hbox{Code} (({x}, {y}, m,k), 1).\]
In that situation, we enlarge the $E$-set as follows. We let $(\alpha_{\beta}, \dot{y}, m, k)=: (\alpha_{\beta}, \dot{y}_{\alpha_{\beta}}, m_{\alpha_{\beta}}, k_{\alpha_{\beta}})$ and \[E_{\alpha_{\beta}+1}:= E_{\alpha_{\beta}} \cup \{ (\alpha_{\beta}, \dot{y}, m, k) \} .\]
Further, if we let $r_{\eta}$ be the real which is added by  $\hbox{Code} (({x}, {y}, m,k), 1)$ at stage $\eta$ of the iteration which witnesses $({\ast} {\ast} {\ast})$ of some quadruple $(x_{\eta},y_{\eta},m_{\eta},k_{\eta}$). Then we collect all the countably many such reals we have added so far in our iteration up to stage $\beta$ and put them into one set $R$ and let
\[ R_{\alpha_{\beta}+1 }:= R . \] 
\item[(ii)] Assume that (i) is wrong, but there is an $\alpha_{\beta}$-legal forcing $\forceQ$ with respect to $E_{\alpha_{\beta}}$ and some $F'$  in $W[G_{\beta}]$ such that 
\begin{align*}
\forceQ \Vdash 
 x \in A_k(y).
\end{align*}
Then we set
\[\forceP(\beta):= \hbox{Code} (({x}, {y}, m,k), 2).\]
In that situation, we enlarge the $E$-set as follows. We let the new $E$ value $(\alpha_{\beta}, \dot{y}_{\alpha_{\beta}}, m_{\alpha_{\beta}}, k_{\alpha_{\beta}})$ be $ (\alpha_{\beta}, \dot{y}, m, k)$  and \[E_{\alpha_{\beta}+1}:= E_{\alpha_{\beta}} \cup \{ (\alpha_{\beta}, \dot{y}, m, k) \}.\]
Further, if we let $r_{\eta}$ be the real which is added by  $\hbox{Code} (({x}, {y}, m,k), 1)$ at stage $\eta$ of the iteration which witnesses $({\ast} {\ast} {\ast})$ of some quadruple $(x_{\eta},y_{\eta},m_{\eta},k_{\eta}$). Then we collect all the countably many such reals we have added so far in our iteration up to stage $\beta$ and put them into one set $R$ and let
\[ R_{\alpha_{\beta}+1 }:= R  . \] 

\item[(iii)] If neither (i) nor (ii) is true, then there is no $\alpha_{\beta}$-legal forcing $\forceQ$ with respect to $E_{\alpha_{\beta}}$ which forces $x \in A_m(y)$ or $x \in A_k(y)$, and we set
\[ \forceP(\beta):=\hbox{Code} (({x}, {y}, m,k), 1).\]
Further, if we let $r_{\eta}$ be the real which is added by  $\hbox{Code} (({x}, {y}, m,k), 1)$ at stage $\eta$ of the iteration which witnesses $({\ast} {\ast} {\ast})$ of some quadruple $(x_{\eta},y_{\eta},m_{\eta},k_{\eta}$). Then we collect all the countably many such reals we have added so far in our iteration up to stage $\beta$ and put them into one set $R$ and let
\[ R_{\alpha_{\beta}+1 }:= R . \] 
Otherwise we force with the trivial forcing.
\end{itemize}
 \end{itemize}
At limit stages $\beta$, we let $\forceP_{\beta}$ be the inverse limit of the $\forceP_{\eta}$'s, $\eta < \beta$, and set $E_{\alpha_{\beta}}= \bigcup_{\eta < \beta}  E_{\alpha_{\eta}}$.
This ends the definition of $\forceP_{\omega_1}$.

\subsection{Discussion of the resulting universe}
We let $G_{\omega_1}$ be a $\forceP_{\omega_1}$-generic filter over $W$.
As $W[G_{\omega_1}]$ is a proper extension of $W$, $\omega_1$ is preserved. Moreover $\CH$ remains true.
A second observation is that for every stage $\beta$ of our iteration and every $\eta > \beta$, the intermediate forcing $\forceP_{[\beta, \eta)}$, defined as the factor forcing of $\forceP_{\beta}$ and $\forceP_{\eta}$, is always an $\alpha_{\beta}$-legal forcing relative to $E_{\alpha_{\beta}}$ and some bookkeeping. This is clear as by the definition of the iteration, we force at every stage $\beta$ with a $\alpha_{\beta}$-legal forcing  relative to $E_{\alpha_{\beta}}$ and $\alpha_{\beta}$-legal becomes a stronger notion as we increase $\alpha_{\beta}$.

We shall define the separating sets now.
For a pair of disjoint $\Sigma^1_3(y)$ sets $A_m(y)$ and $A_k(y)$
we consider the least stage $\beta$ such that there is a $\forceP_{\beta}$-name $\dot{z}$ such that $\dot{z}^{G_{\beta}}=z$ and $(z,y,m,k)$ are considered by $F$ at stage $\beta$. Let $R_{\beta}$ be the set of all reals  which were added by the coding forcing up to stage $\beta$ and which witness $({\ast}{\ast}{\ast})$ for some $(x,y,m,k)$.
Then for any real $x \in W[G_{\omega_1}]$:
\begin{align*}
x \in D^1_{y,m,k} (R_{\beta})\Leftrightarrow \exists r \notin R_{\beta} (&L[r]  \models (x,y,m,k) \text{ can be read off from a code} \\& \text{ written on an } \omega_1\text{-many $\omega$-blocks of elements of } \\& \vec{S^1} ).
\end{align*}
and
\begin{align*}
x \in D^2_{y,m,k} (R_{\beta})\Leftrightarrow \exists r \notin R_{\beta} (&L[r]  \models (x,y,m,k) \text{ can be read off from a code} \\& \text{ written on an } \omega_1\text{-many $\omega$-blocks of elements of } \\& \vec{S^2}).
\end{align*}
It is clear from the definition of the iteration that for any  real parameter $y$ and any $m,k \in \omega$, $D^1_{y,m,k} \cup D^2_{y,m,k} = 2^{\omega}$. 
The next lemma establishes that the sets are indeed separating.
\begin{lemma}
In $W[G_{\omega_1}]$, let $y$ be a real and let $m,k \in \omega$ be such that $A_m(y) \cap A_k(y)=\emptyset.$ Then there is an real $R$ such that the sets  $D^1_{y,m,k}(R)$ and $ D^2_{y,m,k}(R)$ partition the reals. 
\end{lemma}
\begin{proof}
Let $\beta$ be the least stage such that there is a real $x$ such that $F(\beta) \upharpoonright 4=(\dot{x}, \dot{y},m,k)$ with $\dot{x}^G_{\beta}=x$, $\dot{y}^G_{\beta}=y$. Let $R$ be $R_{\beta}$ and $R_{\beta}$ be as defined above. Then, as $A_m(y)$ and $A_k(y)$ are disjoint in $W[G_{\omega_1}]$, by the rules of the iteration, case B must apply at $\beta$. Assume now for a contradiction, that $D^1_{y,m,k}(R)$ and $ D^2_{y,m,k}(R)$ do have non-empty intersection in $W[G_{\omega_1}]$. Let $z \in D^1_{y,m,k}(R) \cap D^2_{y,m,k}(R)$ and let $\beta' > \beta$ be the first stage of the iteration which sees that $z$ is in the intersection. Then, by the rules of the iteration and without loss of generality, we must have used case B(i) at $\beta$, and case B(ii) at stage $\beta'$. But this would imply, that at stage $\beta$, there is an $\alpha_{\beta}$-legal forcing with respect to $E_{\alpha_{\beta}}$, which forces $x \in A_m(y) \cap A_k(y)$, namely the intermediate forcing $(\forceP_{\beta \beta'} )$. This is a contradiction.

\end{proof}

\begin{lemma}
In $W[G_{\omega_1}]$, for every pair $m,k$ and every parameter $y \in 2^{\omega}$ such that $A_m(y) \cap A_k(y) = \emptyset$ there is a real $R$ such that 
\[ A_m(y) \subset D^1_{y,m,k}(R) \land A_k(y) \subset D^2_{y,m,k}(R)\]
\end{lemma}
\begin{proof}
The proof is by contradiction. Assume that there is a real $x$ such that $x \in A_m(y) \cap D^2_{y,m,k}(R)$ for every $R$.
We consider the smallest ordinal $\beta < \omega_1$ such that  $F(\beta)\upharpoonright 4$ considers a quintuple of the form $(x,y,m,k)$ and let $R= R_{\beta}$. As $A_m(y)$ and $A_k(y)$ are disjoint we know that at stage $\beta$ we were in case B. 
As $x$ is coded into $\vec{S^2}$ after stage $\beta$ and by the last Lemma, Case B(i) is impossible at $\beta$. Hence, without loss of generality we may assume that case B(ii) applies at $\beta$. As a consequence, there is a forcing $\forceQ_2$ which is $\alpha_{\beta}$-legal with respect to $E_{\alpha_{\beta}}$ which forces $\forceQ_2 \Vdash x \in A_k(y)$. Note that in that case we collect all the reals which witness $( {\ast} {\ast} {\ast})$ for some quadruple to form the set $R_{\beta}$. 

As $x \in A_m(y) \cap  D^2_{y,m,k} (R)$, we let $\beta'> \beta$ be the first stage such that $W[G_{\beta'}] \models x \in A_m(y)$. By Lemma \ref{productlegal}, $W[G_{\beta}]$ thinks that $\forceQ_2 \times \forceP_{\beta \beta'}$ is $\alpha_{\beta}$-legal with respect to $E_{\alpha_{\beta}}$, yet 
$\forceQ_2 \times \forceP_{\beta \beta'} \Vdash x \in A_m(y) \cap A_k(y)$. This is a contradiction.
\end{proof}

The next lemma will finish the proof of our theorem:
\begin{lemma}
In $W[G_{\omega_1}]$, if $y \in 2^{\omega}$ is an arbitrary  parameter, $R$ a real and $m,k$ natural numbers, then the sets $D^1_{y,m,k}(R)$ and $D^2_{y,m,k}(R)$ are $\Sigma^1_3(R)$-definable.
\end{lemma}
\begin{proof}
The proof is almost identical to the proof of Lemma \ref{Sigma13}, the only thing added is the real $R$ as parameter.
\end{proof}

\section{Forcing the lightface $\Sigma^1_3$-separation property}

The techniques developed in the previous sections can be used to force a model where the (lightface) $\Sigma^1_3$-separation property is true. In what follows, we heavily use ideas and notation from earlier sections, so the upcoming proof can not be read independently.

\begin{theorem}
Starting with $L$ as the ground model, one can produce a set-generic extension $L[G]$ which satisfies $\CH$ and in which the ${\Sigma^1_3}$-separation property holds.
\end{theorem}

\begin{proof}

For the the proof to come, we will redefine the notion of legal forcings.
\begin{definition}
A mixed support iteration is called (0-)legal if it is defined as in Definition 4.2, with the only difference that we code (reals that code) quadruples of the form $(x,0,m,k)$ and $(x,1,m,k)$, where $m,k \in \omega$ and $x$ is (the name of) a real into $\vec{S}$.
\end{definition}
Similar to before, we take 0-legal forcings as the base set of forcings and define gradually smaller families of forcings, which we call $n$-legal forcings with respect to a set $E$ and a bookkeeping function $F$.
Assume that $E$ is a finite list of length $n$ of pairwise distinct pairs of natural numbers $(m_l, k_l)$, $l \le n$,
and that $F$ is a bookkeeping function.
Assume that for every $l \le n$ we do have a notion of $l$-legality with respect to $E \upharpoonright l$, let $(m_{n+1}, k_{n+1})$ be a new pair of natural numbers, distinct from the previous ones. Then we say that $\forceP$ is $n+1$-legal with respect to $E \cup \{ (m_{n+1}, k_{n+1} ) \}$ and $F$ if for every $l \le n$, $\forceP $ is $l$-legal with respect to $E \upharpoonright l$ and $F$ and it obeys the following rules:
\begin{enumerate}
\item $\forceP$ never codes a quadruple of the form $(x,1,m_{n+1},k_{n+1})$ into $\vec{S}$.
\item Whenever $\beta < \gamma $, where $\gamma$ is the length of the iteration $\forceP$, is such that there is a $\forceP_{\beta}$-name $\dot{x}$ of a real and an integer $ i\in\{1,2\}$ such that
\[F(\beta)= (\dot{x},m_{n+1},k_{n+1},i)\] and for  $G$ which is $\forceP_{\beta}$-generic over $W$, $W[G]$ thinks that
\begin{align*}
\exists \forceQ (&\forceQ \text{ is } n \text{-legal with respect to } E \,  \land \\&   \forceQ \Vdash x \in A_{m_{n+1}}),
\end{align*}
where $x=\dot{x}^G$, and $y_{\alpha}=\dot{y}_{\alpha+1}^G$.
Then continuing to argue in $W[G]$, we let
\[\forceP(\beta)= \hbox{Code}((x,0,m_{n+1},k_{n+1}),1).\] 

\item
Whenever $\beta < \gamma $ is such that there is a $\forceP_{\beta}$-name $\dot{x}$ of a real and an integer $ i\in\{1,2\}$ such that
\[F(\beta)= (\dot{x},m_{n+1},k_{n+1},i)\] and for $G_{\beta}$ which is $\forceP_{\beta}$-generic over $W$, $W[G_{\beta}]$ thinks that
\begin{align*}
\forall \forceQ_1 (&\forceQ_1 \text{ is } n \text{-legal with respect to } E  \\& \rightarrow \, \lnot(\forceQ_1 \Vdash x \in A_{m_{n+1}}))
\end{align*}
but there is a forcing $\forceQ_2$ such that $W[G_{\beta}]$ thinks that 
\begin{align*}
\forceQ_{2} \text{ is } n \text{-legal with respect to $E$ and } \\   {\forceQ_2} \Vdash x \in A_{k_{n+1}}
\end{align*}

Then continuing to argue in $W[G_{\beta}]$, we force with
\[\forceP(\beta):= \hbox{Code}((x,0,m_{n+1},k_{n+1}),2).\]

\item If neither 1 nor 2 is true, then either \[ \forceP(\beta)=\hbox{Code}((x,0,m_{n+1},k_{n+1}),2)\] or 
\[ \forceP(\beta)=\hbox{Code}((x,0,m_{n+1},k_{n+1}),1)\] depending on whether $ i\in\{1,2\}$ in $F(\beta)$ was 1 or 2. Otherwise $\forceP$ uses the trivial forcing at that stage. 

\item 
If $F(\beta) = (\dot{x},m,k,i)$ and for every $\forceP_{\beta}$-generic filter $G$, $W[G] \models \forall l \le n+1 ((m_l,k_l) \ne (m,k))$,
then let
\[ \forceP(\beta)=\hbox{Code}((x,0,m,k),i)\] depending on whether $ i\in\{1,2\}$ in $F(\beta)$ was 1 or 2.  
\end{enumerate} 
This ends the definition for the successor step $n \rightarrow n+1$.

With this new notion of $n$-legality, we start the proof of the theorem. The ground model over which we form an iteration is the universe $W$ again, which was defined earlier. Over $W$ we will perform first an $\omega$-length, finitely supported iteration $(\forceP_n)_{n \in \omega}$ of legal posets, and then a second legal iteration of length $\omega_1$. The codes of the form $(x,0,m,k)$ shall eventually define the separating sets for $\varphi_m$ and $\varphi_k$; codes of the form $(x,1,m,k)$ shall correspond to countable sets of reals  (i.e. reals themselves) which indicate the correctness of certain codes of the form $(x,0,m,k)$ which avoid the coding areas coded by these reals.

We let $\{(\varphi_{m_n}, \varphi_{k_n}) \, : \, n \in \omega \}$ be an enumeration of all pairs of $\Sigma^1_3$-formulas. Assume that $(\varphi_{m_0}, \varphi_{k_0})$ is such that there is no legal forcing $\forceQ$ such that $W[\forceQ] \models \exists z (\varphi_{m_0}(z) \land \varphi_{k_0}(z))$. Repeating the arguments from before, we set $E_0:=\{m_0,k_0) \}$ and define the notion of 1-legal with respect to $E_0$.
As will become clear in a second, every step of the iteration $(\forceP_n \, : \, {n \in\omega})$ will either use a legal forcing or define a new and gradually stronger notion of legality. We let $l_n \in \omega$ denote the degree of legality we have already defined at stage $n \in \omega$ of our iteration and define $l_0$ to be $0$ (where 0-legal should just be legal) and $l_1$ to be 1 for the base case of our induction; likewise $\forceP_0$ is set to be the trivial forcing.

The forcing $\forceP_{\omega}$ is the countably supported iteration of $(\forceP_n \, : \,n \in \omega)$, which we will define inductively.
Assume we are at stage $n  \ge 1 \in \omega$ of the iteration and we have defined already the following list of objects and notions:
\begin{enumerate}
\item $\forceP_{n-1}$ and the generic filter $G_{n-1}$.
\item A natural number $l_n \le n$ and a notion of $l_n$-legal relative to $E_{l_n}= \{ (m'_0,k'_0),...,(m'_{l_n-1},k'_{l_n-1}) \} \subset \{ (m_0,k_0),...,(m_{n-1},k_{n-1}) \} $, which is a strengthening of 1-legal relative to $E_0$.
\item A finite set of reals  $\{R_0<...<R_{n - (l_n -1)} \}$, where each real $R_i$ codes a countable set of reals. The choice of the indices will become clear later.

\end{enumerate}
Consider now the $n+1$-th pair $(\varphi_{m_n}, \varphi_{k_n})$ and split into cases.

\begin{enumerate}
\item[$(a)$] There is an $l_n$-legal forcing $\forceQ$ such that \[W[G_{n-1}] \models \forceQ \Vdash \exists z (\varphi_{m_n} (z) \land \varphi_{k_n} (z)), \] then we must use the forcing $\forceQ$. We collect all the reals we have added so far generically which witness $({\ast} {\ast} {\ast})$ for a triple $(x,m,k)$ and call the set $R_{n-l_n}$. In a second step, we use the usual method to code the quadruple $(R_{{n-l_n}}, 1,m'_{l_{n-1}},k'_{l_{n-1}} )$ into $\vec{S}^1$. 
\item[$(b)$] In the second case there is no $l_{n}$-legal forcing $\forceQ$ relative to $E_{l_n}$ which forces $\varphi_{m_n}$ and $\varphi_{k_n}$ to have non-empty intersection. In that case we force with the trivial forcing and define the notion $l_{n+1}$-legal. We first let $(m'_{l_n},k'_{l_n})=(m_n,k_n)$ and $E_{l_n +1}:= E_{l_n} \cup \{(m_n,k_n)\}$, and define $l_{n+1}$-legal relative to $E_{l_n+1}$ just as above. 
We  do not define a new $R_{n-l_n}$.
\end{enumerate}

We let $\forceP_{\omega}$ be the inverse limit of the forcings $\forceP_n$ and consider the universe $W[\forceP_{\omega}]$.
We shall and will assume from now on that in $\forceP_{\omega}$, case $(b)$ is applied infinitely many times. 
\begin{lemma}
For every $n \in \omega$, the tail of the iteration $(\forceP_m \, :\, m \ge n)$ is at least an $l_n-1$-legal iteration relative to $E_{l_n}$ as seen from $W[G_n]$.
\end{lemma}
\begin{proof}
This can easily  be seen if one stares at the definition of the iteration. 
At stage $n$ we either follow case $(a)$ which is an $l_n -1$-legal forcing, as we use the $l_n$-legal $\forceQ$ and additionally code $(R_{{n-l_n}}, 1,m'_{l_{n-1}},k'_{l_{n-1}} )$ which results in an $l_n-1$-legal forcing. Or we apply case $(b)$, thus define $l_n+1$-legality and every further factor of the iteration must be $l_n$-legal. As mixed support iterations of $l_{n}-1$-legal forcings yield an $l_n -1$-legal forcing, this ends the proof.

\end{proof}

The arguments which are about to follow will depend in detail heavily on the actual form of the iteration $\forceP_{\omega}$ which in turn depends on how the enumeration of the $\Sigma^1_3$-formulas behaves. The theorem will of course be true, no matter how $\forceP_{\omega}$ does look like. To facilitate the arguments and notation, however, we assume without loss of generality from now on that the sequence $(\varphi_{m_n}, \varphi_{k_n})$ is so chosen that in the definition of $\forceP_{\omega}$ the cases $(a)$ and $(b)$ are alternating. Thus whenever $n$ is even then $(\varphi_{m_n}, \varphi_{k_n})$ is such that case $(b)$ has to be applied and whenever $n$ is odd then for $(\varphi_{m_n}, \varphi_{k_n})$ case $(a)$ is the one which one should apply. Via changing the order of $(\varphi_{m_n}, \varphi_{k_n})$, this can always be achieved. Consequentially, at even stages $2n$ of the iteration, we define the new notion of $n+1$-legal, while forcing with the trivial forcing; on odd stages $2n+1$, we use the $n+1$-legal forcing to force that $\varphi_{m_{2n+1}}$ and $\varphi_{k_{2n+1}}$ have non-empty intersection and then form the real $R_{2n+1 -n} =R_{n+1}$ and then  code the quadruple $(R_{n+1},1,m_{2n},k_{2n})$ into $\vec{S^1}$.

A consequence of the last lemma (and our assumption on the form of $\forceP_{\omega}$) is that for every even natural number $2n$, there is a final stage in $\forceP_{\omega}$ where we create new codes of the form $(R,1,m_{n},k_{n})$. Indeed, by the definition of $l_{n+1}$-legal, no codes of the form $(R,1,m_{n},k_{n})$ are added by $l_{n+1}$-legal forcings and we have that
\begin{align*}
R_{n}:=  \{ R \, : \,  (R, 1,m_{2n},k_{2n}) \text{ is coded into } \vec{S^1} \}.
\end{align*}
To introduce a useful notion, we say that a real $x$, which is coded into $\vec{S^1}$ or $\vec{S^2}$, has coding area almost disjoint from the real $R$, if $R$ codes an $\omega_1$-sized subset of $\omega$ and the $\omega$-blocks, where $x$ is coded are almost disjoint from the set of ordinals coded by $R$, in that their intersection is countable.

\begin{lemma}
In $W[\forceP_{\omega}]$ for every $n \in \omega \backslash 0$, there is only one real $R$, namely $R_{n}$ which has a code of the form $(R_n,1,m_{2n},k_{2n})$ written into $\omega_1$-many  $\omega$-blocks of elements of $\vec{S^1}$ almost disjoint from $R_{n-1}$.
\end{lemma}
\begin{proof}
This is just a straightforward consequence of the definition of the iteration and of our assumption on the form of $\forceP_{\omega}$.  For $n=1$, note that $R_0$ is the unique ordinal for which $(R,1,m_0,k_0)$ is coded into $\vec{S^1}$. Then the next forcing $\forceP(2)$ is trivial while 2-legal is defined and the forcing $\forceP(3)$ first forces $\varphi_{m_3}$ and $\varphi_{k_3}$ to intersect without creating any code of the form $(R,1,m_0,k_0)$ or $(R,1,m_2,k_2)$, forms $R_1$ and writes $(R_1,1,m_2,k_2)$ into $\vec{S}^1$ with coding area almost disjoint from $R_0$. As all later factors of the iteration are 2-legal, there will be no new codes of the form $(R,1,m_2,k_2)$, hence $R_1$ is the unique real for which $(R_1,1,m_2,k_2)$ is written into $\vec{S^1}$ withwith coding area almost disjoint from $R_0$.

The argument for arbitrary $n$ works exactly the same way with the obvious replacements of letters.

\end{proof}
\begin{lemma}
In $W[\forceP_{\omega}]$, every real  $R_{n}$ is $\Sigma^1_3$-definable.
\end{lemma}
\begin{proof}
This is by induction on $n$. For $n=0$, $R_0$ is the unique real  for which $(R,1,m_0,k_0)$ is coded into $\vec{S^1}$.
This can be written in a $\Sigma^1_3$-way:
\begin{align*}
x= R_0 \Leftrightarrow \exists r \forall M &(|M|=\aleph_0 \land M \text{ is transitive } \land r \in M \land \omega_1^M=(\omega_1^L)^M \rightarrow \\&M \text{ sees with the help of the coded information in } r \text{ that }\\& x \text{ is the unique real such that } (x,1,m_0,k_0) \text{ is coded in} \\& \text{ a block of } \vec{S^1}).
\end{align*}

Now assume that there is a $\Sigma ^1_3$-formula which uniquely defines $R_{n}$, then, by the last lemma, $R_{n+1}$ is the unique real which has a code of the form $(R,1,m_{2(n+1)},k_{2(n+1)})$ written into $\aleph_1$-many $\omega$-blocks of elements of $\vec{S^1}$ almost disjoint from the set of ordinals coded by $R_n$. Let $\psi$ be the $\Sigma^1_3$-formula which defines $R_{n}$, then
\begin{align*}
x= \alpha_{n+1} \Leftrightarrow  &\psi(R_{n}) \text{ and } \\ &\exists r \forall M (|M|=\aleph_0 \land M \text{ is transitive } \land r \in M \land \omega_1^M=(\omega_1^L)^M \rightarrow \\&M \text{ sees with the help of the coded information in } r \text{ that }\\& x \text{ is the unique real such that } (x,1,m_{2(n+1)},k_{2(n+1)}) \\& \text{is coded in} \text{ $\aleph_1$-many blocks of } \vec{S^1} \\& \text{almost disjoint from the ordinals coded by }  R_{n}).
\end{align*}

\end{proof}
The ordinals $R_{n}$ indicate the set of places we need to exclude in order ro obtain correct codes of the form $(x,0,m_{2n},k_{2n})$ which are written into $\vec{S}$. Indeed the iteration $\forceP_{\omega}$, after the stage where we coded the quadruple $(R_{n},1,m_{2n},k_{2n})$ into $\vec{S^1}$ will be $2n-1$-legal, just by the definition of the iteration, which means that the tail of $\forceP_{\omega}$ will never produce a pathological situation.

Finally we can form our desired universe of the $\Sigma^1_3$-separation property. We let $E_{\omega} = \{(m_{2n},k_{2n}) \,: \, n \in \omega\}$ and force with $W[\forceP_{\omega}]$ as our ground model.  We use a countably supported iteration of length $\omega_1$ where we force every quadruple of the form $(x,0,m_{2n},k_{2n})$, $x \in 2^{\omega}$, into either $\vec{S^1}$ or $\vec{S^2}$ with coding area almost disjoint from $ R_{n}$, according to whether case 2, 3 or 4 is true in the definition of $n$-legal. Note that, by assumption, all pairs $(\varphi_{m_n}, \varphi_{k_n})$, $n$ odd, do have a non-empty intersection.  Let $W_1$ denote the universe we obtain this way.

As a consequence we can define in $W_1$ the desired separating sets as follows.
For a pair $(\varphi_{m_{2n}},\varphi_{k_{2n}})$ we let $\psi_{n}$ be the $\Sigma^1_3$-formula which defines $\alpha_{n}$. Now for any real $x$  we let 
\begin{align*}
x \in D_{m_{2n},k_{2n}} \Leftrightarrow \, &\psi(R_{n}) \text{ and } \\& \exists r \forall M (|M|=\aleph_0 \land M \text{ is transitive } \land r \in M \land \omega_1^M=(\omega_1^L)^M  \\&\rightarrow M \text{ sees with the help of the coded information in } r \text{ that }\\& (x,0,m_{2n},k_{2n})  \text{ is coded into $\aleph_1$-many blocks of } \vec{S^1} \\& \text{almost disjoint from the ordinals coded by }  R_{n}).
\end{align*}
and
\begin{align*}
x \notin D_{m_{2n},k_{2n}} \Leftrightarrow \, &\psi(R_{n}) \text{ and } \\& \exists r \forall M (|M|=\aleph_0 \land M \text{ is transitive } \land r \in M \land \omega_1^M=(\omega_1^L)^M  \\&\rightarrow M \text{ sees with the help of the coded information in } r \\& \text{that } (x,0,m_{2n},k_{2n})  \text{ is coded into $\aleph_1$-many blocks of } \vec{S^2} \\&\text{almost disjoint from the ordinals coded by }  R_{n}).
\end{align*}
Both formulas are $\Sigma^1_3$, hence the $\Sigma^1_3$-separation property holds in $W_1$.

\end{proof}

\section{Possible further applications and open problems}

The method which was used to prove the consistency of $\bf{\Sigma^1_{3}}$-separation can be applied to the generalized Baire space as well as we will sketch briefly. Let $BS(\omega_1)$ be defined as $\omega_1^{\omega_1}$ equipped with the usual product topology, i.e. basic open sets are of the form $O_{\sigma} :=\{ f \supset \sigma \, : \, f \in  \omega_1^{\omega_1}, \, \sigma \in \omega_1^{\omega} \}$. The projective hierarchy of $BS(\omega_1)$ is formed just as in the classical setting via projections and complements. The $\bf{\Sigma^1_{1}}$-sets are projections of closed sets, the $\bf{\Pi^1_{1}}$-sets are the complements of the $\bf{\Sigma^1_{1}}$-sets and so on.

The corresponding separation problem in $BS(\omega_1)$ is the following: does there exist a set generic extension of $L$ where $\bf{\Sigma^1_1}$-sets can be separated with $\bf{\Delta^1_1}$-sets?
Our above proof can be applied here as well. All we have to do is to lengthen our sequence of stationary sets we will use to code.

We start with $L$ as our ground model, fix our definable sequence of pairwise almost disjoint, $L$-stationary subsets of $\omega_1$, $(R_{\alpha} \, : \, \alpha < \omega_2)$.  We again split $\vec{R}$ into $\vec{R}^1$ and $\vec{R}^2$, add $\omega_2$-many Suslin trees $\vec{S}$ generically and use $\vec{R}$ to code up $\vec{S}$ just as we did it in the construction of the universe $W$, but leave out the almost disjoint coding forcings as we quantify over $H(\omega_2)$ anyway.
Next we list the $\bf{\Sigma^1_1}$-formulas $\varphi_n$ and start an $\omega_2$-length iteration where we add branches of members of the definable $(S_{\alpha} \, : \, \alpha < \omega_2)$ whenever our bookkeeping function $F$ hands us a triple $(x,m,k)$, just as in the situation of the usual Baire space. As there we distinguish the several cases and restrict ourselves to \emph{legal} forcings, where legal is the straightforward adjustment of legal in the $\omega$-case. The separating sets $D_{m,k}$ are defined using $\aleph_1$-sized, transitive models as which witness the wanted patterns on $\vec{S}^1$ and $\vec{S}^2.$

The sequence of the fixed $W$-Suslin trees $(S_{\alpha} \, : \, \alpha < \omega_1)$ is $\Sigma_1 (\omega_1)$-definable, thus the codes we write into them are $\Sigma_1(\omega_1)$-definable as well. We do not have to add almost disjoint coding forcings, as we quantify over subsets of $\omega_1$ in this setting anyway. All the factors will have the ccc, thus an iteration of length $\omega_2$ is sufficient to argue just as above that in the resulting generic extension $L[G]$, every pair of $\bf{\Sigma^1_1}$-sets is separated by the according $D_{m,k}$.

The just sketched method is not limited to the case $\omega_1$. Indeed, if $\kappa$ is a successor cardinal,in $L$ then we can lift the argument to $\kappa$ as well. The proof will rely on a different kind of preservation result for iterated forcing constructions, as we can not use Shelah's theory of iterations of $S$-proper forcings anymore. Also, the choice of the definable sequence of $L$-stationary subsets of $\kappa$ has to be altered slightly, as we can not shoot clubs in a nice way through arbitrary stationary subsets of $\kappa$. How to solve some of the just posed problems is worked out in \cite{SyLL}.

What remains an interesting open problem is the following:
\begin{question}
Can one force the $\Sigma^1_{1}$-separation property for $BS(\kappa)$ where $\kappa$ is inaccessible? What if $\kappa$ is weakly compact?
\end{question}

Another possible further direction is, as mentioned already in the introduction, to replace the ground model $L$ with inner models with large cardinals.
We expect that a modification of the ideas of this article can be lifted in that context. In particular we expect that, given any natural number $n \ge 1$, over the canonical inner model with $n$-many Woodin cardinals $M_n$ one can force a model for which the $\Sigma^1_{n+3}$-separation property is true.
Note here, that this would produce universes where the $\Sigma^1_{2n}$-separation property is true for the first time.
These considerations rely on large cardinals however and it is interesting whether one can get by without them.

\begin{question}
For $n \ge 4$, can one force the  $\Sigma^1_n$-separation property over $L$?
\end{question}
Last, note that the technique presented in this article seems to only work locally for one fixed $\Sigma^1_n$-pointclass. It would be very interesting to produce a model where we force a global behaviour of the $\Sigma^1_n$-separation property.

\begin{question}
For $n,m \in \omega$, is it possible to force the $\Sigma^1_n$ and the $\Sigma^1_m$-separation property simultaneously?
If $E \subset \omega$, is it possible to force a universe in which the $\Sigma^1_n$-property is true for every $n \in E$?
\end{question}
Last, it is tempting to analyse how much consequences of $\Delta^1_2$-determinacy can be forced to hold simultaneously.
A first test question in that direction would be
\begin{question}
Is it possible to force over $L$ the existence of a universe in which the $\Sigma^1_3$-separation property holds and every $\Sigma^1_3$-set has the Baire property? 
\end{question}

\end{document}